\documentclass[reqno,11pt]{amsart}

\usepackage{epsfig,amssymb,amsfonts,amsmath,graphics,color}
\usepackage{latexsym}		
\usepackage{a4wide}			
\usepackage[latin1]{inputenc}	
\usepackage{dsfont} 
\usepackage[pdfstartview=XYZ]{hyperref} 

\allowdisplaybreaks[4] 

\theoremstyle{plain}				
\newtheorem  {theorem}                  {Theorem}       [section]
\newtheorem  {lemma}        [theorem]   {Lemma}

\newtheorem  {corollary}     [theorem]   {Corollary}
\newtheorem  {proposition}  [theorem]   {Proposition}
\theoremstyle{definition}			
\newtheorem  {definition}   [theorem]   {Definition}
\newtheorem  {example}     [theorem]   {Example}

\newtheorem  {remark}    [theorem]   {Remark}

  \numberwithin{equation}{section}	
  \numberwithin{figure}{section}	

\newcommand{\notion}[1]{\emph{#1}} 
\newcommand{\notionOI}[1]{\emph{#1}} 
\newcommand{\notiono}[1]{\emph{#1}} 

\newcommand{\wo}{\backslash}
\newcommand{\borel}{\mathcal{B}} 
\newcommand{\real}{\mathbb{R}} 
\newcommand{\reald}{\mathbb{R}^d} 
\newcommand{\realdc}{\bar{\mathbb{R}}^d} 
\newcommand{\realdcon}{\realdc \wo \{0\}} 
\newcommand{\rational}{\mathbb{Q}} 
\newcommand{\nat}{\mathbb{N}} 
\newcommand{\units}[1][d]{S_{#1}} 
\newcommand{\unitsdd}{\units[\dd]} 
\newcommand{\sdd}{\unitsdd} 
\newcommand{\LL}{\mathbb{L}} 
\newcommand{\dd}{\mathbb{D}} 

\newcommand{\ddcon}{\bar{\dd}_0} 

\newcommand{\rvdd}[1][\mu]{RV_{\ddcon} (\alpha, (a_n), #1)} 
\newcommand{\rvan}[1][\mu]{RV (\alpha, (a_n), #1)} 

\newcommand{\pspace}{(\Omega, \mathcal{F}, P )}

\newcommand{\ev}{\mathbb{E}\,} 
\newcommand{\id}{\mathds{1}} 

\newcommand{\gegen}[1][]{\xrightarrow{#1}}

\newcommand{\gegend}{\gegen[d]}
\newcommand{\gegenv}{\gegen[v]}

\newcommand{\what}{\hat{w}}
\newcommand{\tow}{\gegen[\what]}

\newcommand{\tov}{\gegenv}

\newcommand{\schl}[1]{\widetilde{#1}} 

\renewcommand{\bar}[1]{\overline{#1}} 

\newcommand{\zweiquad}{\quad\quad} 
\newcommand{\vierquad}{\quad\quad\quad\quad} 
\newcommand{\achtquad}{\quad\quad\quad\quad\quad\quad\quad\quad} 

\def\B_e{B_{\eta}(e)}

\definecolor{blue}{rgb}{.1, 0.1,.8}
\definecolor{green}{rgb}{0,0.8,0.2}
\definecolor{red}{rgb}{1, 0,0}

\newcommand{\beao}{\begin{eqnarray*}}
\newcommand{\eeao}{\end{eqnarray*}\noindent}

\newcommand{\beam}{\begin{eqnarray}}
\newcommand{\eeam}{\end{eqnarray}\noindent}

\newcommand{\beqq}{\begin{equation}}
\newcommand{\eeqq}{\end{equation}\noindent}

\newcommand{\bce}{\begin{center}}
\newcommand{\ece}{\end{center}}

\newcommand{\barr}{\begin{array}}
\newcommand{\earr}{\end{array}}

\newcommand{\vague}{\stackrel{\lower0.2ex\hbox{$\scriptscriptstyle
                    \it{v} $}}{\rightarrow}}
\newcommand{\weak}{\stackrel{\lower0.2ex\hbox{$\scriptscriptstyle
                    \it{w} $}}{\rightarrow}}

\newcommand{\bdis}{\begin{displaymath}}
\newcommand{\edis}{\end{displaymath}\noindent}

\begin{document}

\title[Functional Regular Variation of MMA Processes]
{Functional Regular Variation of L\'evy-Driven Multivariate Mixed Moving Average Processes}

\author[Martin Moser]{Martin Moser}
\author[Robert Stelzer]{Robert Stelzer}

\address{Institute for Advanced Study \& Zentrum Mathematik, Technische Universit\"at M\"unchen,
Boltzmann\-stra\ss e 3, 85748 Garching bei M\"unchen, Germany}
\email{moser@ma.tum.de}
\address{Institute of Mathematical Finance, Ulm University, Helmholtzstraße 18, 89081 Ulm, Germany}
\email{robert.stelzer@uni-ulm.de}

\begin{abstract}
  We consider  the functional regular variation in the space $\dd$ of c\`adl\`ag functions of multivariate mixed moving average (MMA) processes of the type $X_t = \int\int f(A, t - s) \Lambda (d A, d s)$. We give sufficient conditions for an MMA process $(X_t)$ to have c\`adl\`ag sample paths. As our main result, we prove that $(X_t)$ is regularly varying in $\dd$ if the driving L\'evy basis is regularly varying and the kernel function $f$ satisfies certain natural (continuity) conditions. Finally, the special case of supOU processes, which are used, e.g., in applications in finance, is considered in detail.
\end{abstract}
\subjclass[2000]{Primary: 60G51; 60G70}
\keywords{ càdlàg sample paths, functional regular variation, heavy tails, L\'evy basis, mixed moving average process}
\thanks{Martin Moser gratefully acknowledges the support of both, the Technische Universit\"at M\"unchen - Institute for Advanced Study, funded by the German Excellence Initiative, as well as the International Graduate School of Science and Engineering (IGSSE) at Technische Universit\"at M\"unchen, Germany.}
\maketitle

\section{Introduction}\label{sec:dd:intro}

In many applications of stochastic processes, the center of the distributions involved and related quantities (e.g.\ sample means, variances etc.) can be modeled quite well. In view of the central limit theorem, Gaussian distributions play an important role in that field. However, this needs not to be true for the tail of the distribution which is of great importance in many areas of application. Possible examples are severe crisis in stock markets or extreme weather events which can cause huge losses to the financial industry, insurance companies and also to private people. Therefore, it is of great importance to model the distribution tail and related quantities (e.g.\ quantiles, exceedances, maxima etc.) correctly.\par
A very well established concept to model extreme values is regular variation. It has its origin in classical extreme value theory, where limit distributions for sample maxima are derived. The maximum domains of attraction of two of the three possible standard extreme value distributions (Fr\'echet and Weibull) can be described by regular variation of functions, meaning functions behaving like a power law in the limit, see also \cite{embrechts:klueppelberg:mikosch:1997} and \cite{resnick:2007}.\par 
Moreover, regular variation can intuitively be extended to a multivariate setup. It is then formulated in terms of vague convergence of measures given by
\begin{equation}\label{eqn:regvarintro}
n P ( a_n^{-1} X \in \cdot ) \tov \mu (\cdot ),
\end{equation}
where $X$ is a multivariate random vector, $(a_n)$ an increasing sequence and $\mu$ is a Radon measure. Since $\mu$ is homogeneous, multivariate regular variation of $X$ can be interpreted as convergence of the radial part $\|X\|$ to a univariate regularly varying random variable $Y$ and of the spherical part $X / \|X\|$ to a random variable $Z$ on the unit sphere, which is independent of $Y$ and can be described by the measure $\mu$. Detailed introductions to multivariate regular variation can be found in \cite{resnick:2007} and \cite{hult:lindskog:2006b}.\par
Finally, \cite{hult:lindskog:2005} extended the definition \eqref{eqn:regvarintro} to the space of multivariate stochastic processes with sample paths in the space $\dd$ of c\`adl\`ag functions, i.e.\ right-continuous functions with limits from the left. The formulation of regular variation in such generality has the advantage that, in addition to functionals based on the values of a stochastic process at fixed time points, one can also analyze functionals acting on the complete sample paths of the process. This is a very powerful tool for the analysis of extremal properties of a process, especially in combination with methods for weak convergence of point processes which are closely related to regular variation (see Section \ref{sec:dd:pointprocess}). Despite the power of this technique, conditions ensuring regular variation in $\dd$ have  so far been given only for few classes of processes.\par
In this paper we apply the concept of regular variation on $\dd$ to multivariate mixed moving average (MMA) processes with c\`adl\`ag sample paths. MMA processes have been first introduced by \cite{surgailis:rosinski:mandrekar:cambanbis:1993} in the univariate stable case and are given as integrals of the form
\[ X_t = \int\limits_{M_d^-} \int\limits_{\real} f(A, t - s) \Lambda (dA, ds),\]
where $\Lambda$ is a multivariate L\'evy basis. The class of (multivariate) MMA processes covers a wide range of processes which are well known and extensively used in applications. Examples include Ornstein-Uhlenbeck processes (cf.\ \cite{barndorff:shephard:2001} and \cite{pigorsch:stelzer:2009}), superpositions of Ornstein-Uhlenbeck (supOU) processes (cf.\ \cite{barndorff:2001} and \cite{barndorff:stelzer:2011a}), (fractionally integrated) CARMA processes (cf.\ \cite{brockwell:2004}, \cite{marquardt:2007}) and increments of fractional L\'evy processes (cf.\ \cite{marquardt:2006}, \cite{bender:lindner:schicks:2011} and references therein).\par
Regular variation of the finite-dimensional distributions of MMA processes has already been proved in \cite{moser:stelzer:2011}, given that the underlying L\'evy basis is regularly varying and the kernel function satisfies the integrability condition $f \in \mathbb{L}^\alpha$. In this paper we give additional integrability and continuity conditions on the kernel function $f$ such that the MMA process is functionally regularly varying on $\dd$. Furthermore, we also analyze the special case of multivariate supOU processes given by
\[ X_t = \int\limits_{M_d^-} \int\limits_{-\infty}^t e^{A (t - s)} \Lambda (dA, ds)\]
and give some more accessible sufficient conditions for supOU processes to be functionally regularly varying.\par
The paper is organized as follows. In Section \ref{sec:dd:prelim:multregvar} we introduce the notion of multivariate regular variation and related properties. In Section \ref{sec:dd:prelim:mma} we recall the definition of MMA processes and the related integration theory. Furthermore, we review the conditions for existence of MMA processes and for regular variation of their finite dimensional distributions. The sample path behavior of MMA processes is discussed in Section \ref{sec:dd:sample}. We give an overview over the relevant literature and derive new sufficient conditions for MMA processes to have c\`adl\`ag sample paths in the case where the underlying L\'evy process is of finite variation. In Section \ref{sec:dd:funcregvar} we introduce the notion of functional regular variation and prove that MMA processes are regularly varying on $\dd$, given certain conditions. In Section \ref{sec:dd:supOU} we verify these conditions in the special case of supOU processes. Finally, in Section \ref{sec:dd:pointprocess} we show the connection between functional regular variation and point process convergence and discuss the relevance of the results to the extremal analysis of MMA processes.

\pagebreak
\section{Preliminaries}\label{sec:dd:prelim}
\subsection{Notation}\label{sec:dd:prelim:not}

Let $\real$ be the real numbers, $\real^+$ the positive and $\real^-$ the negative real numbers, both without $0$. $\nat$ is the set of positive integers and $\rational$ are the rational numbers. The Borel sets are denoted by $\borel$ and $\borel_b$ are the bounded Borel sets. $\lambda$ is the Lebesgue measure on $\real$ and $B_r(x) := \{y \in \reald: \|y - x\| \leq r\}$ is the closed ball of radius $r$ centered at $x$. $\dd$ is the space of c\`adl\`ag (right-continuous with left limits) functions $x : [0,1] \to \reald$ and $\sdd = \{x \in \dd: \sup_{t \in [0,1]} \|x_t\|= 1\}$ is the unit sphere in $\dd$.\par
For matrices, $M_{n,d}$ is the set of all $n \times d$ matrices and $M_d$ the set of all $d\times d$ matrices. $M_d^-$ is the set of all $d\times d$ matrices with all eigenvalues having strictly negative real part. $I_d$ is the $d \times d$ identity matrix. We write $A^T$ for the transposed of a matrix $A$ and $\|A\|$ for the matrix norm induced by the Euclidean norm.\par
If random variables, vectors, processes, measures etc.\ are considered, they are given as measurable mappings with respect to a complete probability space $\pspace$.\par
Vague convergence is defined in terms of convergence of Radon measures and it is denoted by $\gegenv$. It is defined on the one-point uncompactification $\realdcon$, which assures that the sets bounded away from the origin can be referred to as the relatively compact sets in the vague topology. Similarly, $\hat{w}$-convergence is given by the convergence of boundedly finite measures and is defined on $\ddcon = (0, \infty ] \times \sdd$, which can be viewed as the one-point uncompactification in $\dd$.

\subsection{Multivariate Regular Variation}\label{sec:dd:prelim:multregvar}
 Regular variation on $\reald$ is expressed in terms of vague convergence of measures and several different, but equivalent, definitions exist. For detailed and very good introductions to regular variation we refer to \cite{bingham:goldie:teugels:1987}, \cite{resnick:1987}, \cite{resnick:2007} and \cite{lindskog:2004}.

\begin{definition}[Multivariate Regular Variation]\label{def:dd:regvar1}
A random vector $X\in \reald$ is \notionOI{regularly varying} if there exists a sequence $(a_n)_{n \in \nat}$, $0 < a_n \nearrow \infty$, and a nonzero Radon measure $\mu$ on $\borel ( \realdcon )$ such that $\mu ( \realdc \wo \reald ) = 0$ and, as $n \to \infty$,
\[ n P( a_n^{-1} X \in \, \cdot ) \gegenv \mu (\cdot)\]
on $\borel ( \realdcon )$. Similarly, we call a Radon measure $\nu$ regularly varying if $(a_n)$ and $\mu$ exist as above such that, as $n \to \infty$,
\[n\, \nu(a_n^{-1} \cdot) \gegenv \mu(\cdot).\]
\end{definition}

The limiting measure $\mu$ of the definition is homogeneous, i.e.\ it necessarily satisfies the condition
\[\mu (t B) = t^{-\alpha} \mu (B)\]
for all $B \in \borel ( \realdcon )$ and $t >0$. Hence, we write $X \in \rvan$ or $\nu \in \rvan$, respectively. In the special case of an infinitely divisible random vector $X \in \reald$ with Lévy measure $\nu$ we know that $X \in \rvan$ if and only if $\nu \in \rvan$ (see \cite{hult:lindskog:2006a}, Proposition 3.1). This result is very useful throughout this work, since MMA processes are infinitely divisible, just as the driving L\'evy bases are. A detailed introduction to infinitely divisible distributions and L\'evy processes can be found in \cite{sato:2002}, for instance.

\subsection{Multivariate Mixed Moving Average Processes}\label{sec:dd:prelim:mma}

In this section we shortly recall the definition and main properties of multivariate \notionOI{mixed moving average processes} (short MMA processes).\par
A multivariate ($\real^n$-valued) MMA process $(X_t)$ can be defined as an integral over a measurable kernel function $f : M_d^- \times \real \mapsto M_{n,d}$ with respect to an $\reald$-valued L\'evy basis $\Lambda$ on $M_d^- \times \real$, i.e.
\[ X_t := \int \limits_{M_d^-} \int \limits_{\real} f(A, t - s) \Lambda (d A, d s).\]
An $\reald$-valued \notion{L\'evy basis} $\Lambda = (\Lambda (B) )$ with $B \in \borel_b (M_d^- \times \real)$ is a random measure which is

\begin{list}{$\bullet$}{\setlength{\leftmargin}{0.5cm}}
\item \notionOI{infinitely divisible}, i.e.\ the distribution of $\Lambda (B)$ is infinitely divisible for all $B \in  \borel_b (M_d^- \times \real)$,
\item \notionOI{independently scattered}, i.e.\ for any $n \in \nat$ the random variables $\Lambda(B_1), \dots, \Lambda(B_n)$ are independent for pairwise disjoint sets $B_1, \dots, B_n \in \borel_b (M_d^- \times \real)$ and
\item \notionOI{$\sigma$-additive}, i.e.\ for any pairwise disjoint sets $(B_i)_{i \in \nat}  \in \borel_b (M_d^- \times \real)$ with $\bigcup_{n \in \nat} B_n \in \borel_b (M_d^- \times \real)$ we have $\Lambda (\bigcup_{n \in \nat} B_n) = \sum_{n \in \nat} \Lambda(B_n)$ almost surely.
\end{list}

Thus, L\'evy bases are also called infinitely divisible independently scattered random measures (i.d.i.s.r.m.). Following the relevant literature (cf. \cite{fasen:2005}, \cite{fasen:2009}, \cite{fasen:klueppelberg:2007}, \cite{barndorff:stelzer:2011a} and \cite{moser:stelzer:2011}) we only consider time-homogeneous and factorisable L\'evy bases, i.e.\ L\'evy bases with characteristic function
\begin{equation}\label{eqn:pp:charfunc}
\ev \left( e^{iu^T \Lambda(B)} \right) = e^{\varphi ( u ) \Pi (B)}
\end{equation}
for all $u \in \reald$ and $B \in \borel_b (M_d^- \times \real)$, where $\Pi = \lambda \times \pi$ is the product of a probability measure $\pi$ on $M_d^-(\real)$ and the Lebesgue measure $\lambda$ on $\real$ and
\[\varphi ( u ) = iu^T\gamma - \frac{1}{2} u^T\Sigma u + \int\limits_{\reald} \left( e^{iu^Tx} - 1 - iu^T x \id_{[-1,1]} (\|x\|) \right) \nu (dx)\]
is the characteristic function of an infinitely divisible distribution with characteristic triplet $(\gamma, \Sigma, \nu)$.  The distribution of the L\'evy basis is then completely determined by $(\gamma, \Sigma, \nu, \pi)$ which is therefore called the \notionOI{generating quadruple}. By $L$ we denote the \notionOI{underlying Lévy process} which is given by $L_t = \Lambda (M_d^- \times (0, t])$ and $L_{-t} = \Lambda (M_d^- \times [ -t, 0))$ for $t \in \real^+$. For more details on L\'evy bases see \cite{rajput:rosinski:1989} and \cite{pedersen:2003}.\par
We should stress that the set $M_d^-$ in the definition of MMA processes can be replaced by $M_d$ or basically any other topological space. The choice of $M_d^-$ is motivated by the special case of supOU processes, where this is the canonical choice.\par
Necessary and sufficient conditions for the existence of MMA processes are given by the multivariate extension of Theorem 2.7 in \cite{rajput:rosinski:1989}.

\begin{theorem}\label{thm:dd:MMAex}
Let $\Lambda$ be an $\reald$-valued Lévy basis with characteristic function of the form \eqref{eqn:pp:charfunc} and let $f : M_d^- \times \real \mapsto M_{n,d}$ be a measurable function. Then $f$ is $\Lambda$-integrable as a limit in probability in the sense of \cite{rajput:rosinski:1989} if and only if
\small\begin{align}
&\int\limits_{M_d^-} \int\limits_\real \bigg\| f(A,s)\gamma + \int\limits_{\reald} f(A,s) x \left(\id_{[0,1]} \left(\|f(A,s)x\|\right) - \id_{[0,1]} \left(\|x\|\right) \right) \nu (dx) \bigg\| ds \pi (dA) < \infty,\label{eqn:dd:exbed1}\\
&\int\limits_{M_d^-} \int\limits_{\real} \| f(A,s) \Sigma\,f(A,s)^T \| d s \pi (d A) < \infty\quad\mbox{ and}\label{eqn:dd:exbed2}\\
&\int\limits_{M_d^-} \int\limits_{\real} \int\limits_{\reald} \left(1 \wedge \|f(A,s)x\|^2\right) \nu(d x) d s \pi (d A) < \infty.\label{eqn:dd:exbed3}
\end{align}\normalsize
If $f$ is $\Lambda$-integrable, the distribution of $X_0 = \int_{M_d^-}\int_{\real^+} f(A,s) \Lambda (dA, ds)$ is infinitely divisible with characteristic triplet $(\gamma_{int}, \Sigma_{int}, \nu_{int})$ given by
\small\begin{align*}
&\gamma_{int} = \int\limits_{M_d^-} \int\limits_\real \left( f(A,s)\gamma + \int\limits_{\reald} f(A,s) x \left(\id_{[0,1]} \left(\|f(A,s)x\|\right) - \id_{ [0,1]} \left(\|x\|\right) \right) \nu (dx) \right) ds \pi (dA),\\
&\Sigma_{int} = \int\limits_{M_d^-} \int\limits_{\real} f(A,s) \Sigma\,f(A,s)^T d s \pi (d A)\quad\mbox{ and}\\
&\nu_{int} (B)= \int\limits_{M_d^-} \int\limits_{\real} \int\limits_{\reald} \id_B(f(A,s)x) \nu(d x) d s \pi (d A) \quad \mbox{ for all Borel sets } B \subseteq \reald.
\end{align*}\normalsize
\end{theorem}

However, since the focus of this section is on regularly varying processes, we will also use some more convenient conditions for this special setting. They have been derived in \cite{moser:stelzer:2011} and are based on integrability conditions described by the following space of functions
\small\[\mathbb{L}^\delta (\lambda \times \pi) := \Bigg\{ f : M_d^- \times \real \mapsto M_{n,d} \mbox{ measurable, } \int\limits_{M_d^-}\int\limits_\real \| f(A,s) \|^\delta ds \pi (dA) < \infty \Bigg\}.\]\normalsize

\begin{theorem}[\cite{moser:stelzer:2011}, Theorem 2.6]\label{thm:dd:MMAex2}
Let $\Lambda$ be a L\'evy basis with values in $\reald$ and characteristic quadruple $(\gamma, \Sigma, \nu, \pi)$. Furthermore, let $\nu$ be regularly varying with index $\alpha$ and let $f : M_d^- \times \real \mapsto M_{n,d}$ be a measurable function. Then $f$ is $\Lambda$-integrable in the sense of \cite{rajput:rosinski:1989} and $X_t$ is well defined for all $t \in \real$, stationary and infinitely divisible with known characteristic triplet (see Theorem \ref{thm:dd:MMAex}) if one of the following conditions is satisfied:
\begin{enumerate}
	\item[(i)] $L_1$ is $\alpha$-stable with $\alpha \in (0,2) \wo \{1\}$ and $f \in \LL^\alpha \cap \LL^1$.
	\item[(ii)] $f$ is bounded and $f \in \LL^\delta$ for some $\delta < \alpha$, $\delta \leq 1$.
	\item[(iii)] $f$ is bounded, $\ev L_1 = 0$, $\alpha > 1$ and $f \in \LL^\delta$ for some $\delta < \alpha$, $\delta \leq 2$.
\end{enumerate}   
\end{theorem}

Regular Variation of $X_t$ for fixed $t \in \real$ as well as regular variation of the finite dimensional distributions of the process $(X_t)$ have been derived in \cite{moser:stelzer:2011} under similar conditions.

\begin{theorem}[\cite{moser:stelzer:2011}, Th. 3.2 and Cor. 3.5]\label{thm:dd:regvarMMA}
Let $\Lambda$ be an $\reald$-valued L\'evy basis on $M_d^- \times \real$ with generating quadruple $(\gamma, \Sigma, \nu, \pi)$ and let $\nu \in \rvan[\mu_\nu]$. If $X_0 = \int_{M_d^-}\int_{\real} f(A,s) \Lambda (dA, ds)$ exists (in the sense of Theorem \ref{thm:dd:MMAex}), $f \in \mathbb{L}^\alpha (\lambda \times \pi)$ and $\mu_\nu ( f^{-1} (A,s) (\real^n \wo \{0\}) ) = 0$ does not hold for $\pi \times \lambda$ almost-every $(A,s)$, then $X_0 \in \rvan[\mu_X]$ with
\[\mu_X (B) := \int\limits_{M_d^-} \int\limits_{\real} \int\limits_{\reald} \id_{B} \left(f(A,s)x\right) \mu_\nu (d x) d s \pi (d A).\]
Furthermore, the finite dimensional distributions $( X_{t_1}, \ldots, X_{t_k} )$, $t_i \in \real$ and $k \in \nat$, are also regularly varying with index $\alpha$ and a given limiting measure $\mu_{t_1, \ldots, t_k}$.
\end{theorem}

Comparable necessary conditions for regular variation do also exist, see \cite{moser:stelzer:2011}, Theorem 3.4, for details.\par
Next we introduce a result which allows to decompose a L\'evy basis into a drift, a Brownian part, a part with bounded jumps and a part with finite variation. This is the extension of the L\'evy-It\^o decomposition to L\'evy bases.

\begin{theorem}[\cite{barndorff:stelzer:2011a}, Theorem 2.2]\label{thm:levyitobases}
Let $\Lambda$ be a L\'evy basis on $M_d^- \times \real$ with characteristic function of the form \eqref{eqn:pp:charfunc} and generating quadruple $(\gamma, \Sigma, \nu, \pi)$. Then there exists a modification $\schl{\Lambda}$ of $\Lambda$ which is also a L\'evy basis with the same characteristic quadruple $(\gamma, \Sigma, \nu, \pi)$ such that there exists an $\reald$-valued L\'evy basis $\schl{\Lambda}^G$ on $M_d^- \times \real$ with generating quadruple $(0, \Sigma, 0, \pi)$ and an independent Poisson random measure $N$ on $\reald \times M_d^- \times \real$ with intensity measure $\nu \times \pi \times \lambda$ such that
\small\[\schl{\Lambda} (B) = \gamma (\pi \times \lambda) (B) + \schl{\Lambda}^G (B) +  \int\limits_{\|x\| \leq 1} \int\limits_B x (N(dx, dA, ds) - \pi (dA) ds \nu (dx)) + \int\limits_{\|x\| > 1} \int\limits_B x N(dx, dA, ds)\]\normalsize
for all $B \in \borel_b (M_d^- \times \real)$ and $\omega \in \Omega$. If, additionally, $\int_{\|x\| \leq 1} \|x\| \nu(dx) < \infty$, then
\[\schl{\Lambda} (B) = \gamma_0 (\pi \times \lambda) (B) + \schl{\Lambda}^G (B) + \int\limits_{\reald} \int\limits_B x N(dx, dA, ds)\]
for all $B \in \borel_b (M_d^- \times \real)$, where
\[\gamma_0 := \gamma - \int_{\|x\| \leq 1} x \nu(dx).\]
Moreover, the Lebesgue integral exists with respect to $N$ for all $\omega \in \Omega$.
\end{theorem}
Throughout the remainder of this paper we assume that all L\'evy bases occurring are already modified such that they have the above L\'evy-It\^o decomposition. Moreover, for a L\'evy basis $\Lambda$ we define two L\'evy bases $\Lambda^{(1)}$ and 
$\Lambda^{(2)}$ by
\begin{align}
 \Lambda^{(1)}(B)&=\gamma (\pi \times \lambda) (B) + \schl{\Lambda}^G (B) +  \int\limits_{\|x\| \leq 1} \int\limits_B x (N(dx, dA, ds) - \pi (dA) ds \nu (dx))\\
 \Lambda^{(2)}(B)&=\int\limits_{\|x\| > 1} \int\limits_B x N(dx, dA, ds)\label{eq:lam2}
\end{align}
for all Borel sets $B$.

In the case of an underlying L\'evy process with finite variation, Theorem \ref{thm:levyitobases} and Theorem \ref{thm:dd:MMAex} can be combined to obtain integrability conditions for this special setting. Note that by Theorem 21.9 in \cite{sato:2002} finite variation of $(L_t)$ is equivalent to $\Sigma = 0$ and $\int_{\|x\| \leq 1} \|x \|\; \nu (dx) < \infty$.

\begin{proposition}[\cite{barndorff:stelzer:2011a}, Prop. 2.4]\label{prop:MMAexfinite}
Let $\Lambda$ be a L\'evy basis on $M_d^- \times \real$ with characteristic function of the form \eqref{eqn:pp:charfunc} and generating quadruple $(\gamma, 0, \nu, \pi)$ such that $\int_{\|x\| \leq 1} \|x \|\; \nu (dx) < \infty$. Let $\gamma_0$ and $N$ be as defined in Theorem \ref{thm:levyitobases}. If $f \in \mathbb{L}^1$ and
\[\int \limits_{M_d^-} \int \limits_{\real} \int \limits_{\reald} (1 \wedge \| f(A, s) x\|) \,\nu (dx)\,ds\,\pi (dA) < \infty,\]
then
\[X = \int \limits_{M_d^-} \int \limits_{\real} f(A,s) \Lambda (dA, ds) = \int \limits_{M_d^-} \int \limits_{\real} f(A,s)\,\gamma_0 \,ds \pi (dA) + \int \limits_{\reald} \int \limits_{M_d^-} \int \limits_{\real}  f(A,s)\,x\, N(dx, dA, ds)\]
and the integrals on the right hand side exist as Lebesgue integrals for every $\omega \in \Omega$. Moreover, the distribution of $X$ is infinitely divisible with characteristic function
\[\ev \left( e^{i u^T X} \right) = \exp \left( i u^T \gamma_{int,0} + \int\limits_{\reald} \left( e^{i u^T x} - 1\right) \nu_{int} (dx) \right),\]
where
\begin{align*}
\gamma_{int,0} &= \int \limits_{M_d^-} \int \limits_{\real} f(A,s)\,\gamma_0 \,ds\,\pi (dA) \quad \mbox{and}\\
\nu_{int} (B) &= \int\limits_{M_d^-} \int\limits_{\real} \int\limits_{\reald} \id_B(f(A,s)x)\, \nu(d x)\, d s\, \pi (d A)
\end{align*}
for all Borel sets $B \subseteq \reald$.
\end{proposition}
The condition $f\in L^1$ is obsolete if $\gamma_0=0$.

%

\section{Sample Path Behavior}\label{sec:dd:sample}

In Section \ref{sec:dd:funcregvar} we review the concept of regular variation for c\`adl\`ag processes and apply it to MMA processes. Therefore, we first have to discuss the sample path behavior of MMA processes.\par
Many examples of results for MMA processes to have c\`adl\`ag sample paths exist in the special case where the underlying L\'evy process has sample paths of finite variation, i.e.\ $\Sigma = 0$ and $\int_{\|x\| \leq 1} \|x \|\; \nu (dx) < \infty$. In this case, the sample path behavior of the driving L\'evy process transfers to the sample paths of the MMA process. For example, define for any L\'evy process $L_t$ the corresponding \notionOI{filtered L\'evy process} $X_t$ by
\begin{equation}\label{eqn:filteredlevy}
X_t = \int\limits_0^t f(t,s)\; dL_s
\end{equation}
for $t \in [0,1]$. If $X_t$ exists, $L_t$ is of finite variation and the kernel function $f$ is bounded and continuous, then $X_t$ has c\`adl\`ag sample paths (cf.\ \cite{hult:lindskog:2005}, Lemma 28).\par
Another result for supOU processes is given by Theorem 3.12 in \cite{barndorff:stelzer:2011a}. This result can be extended to the general case of MMA processes. 

\begin{theorem}\label{thm:cadlagmma}
Let $\Lambda$ be a L\'evy basis on $M_d^- \times \real$ with characteristic function of the form \eqref{eqn:pp:charfunc} and generating quadruple $(\gamma, 0, \nu, \pi)$ such that $\int_{\|x\| \leq 1} \|x \|\; \nu (dx) < \infty$
. Suppose that the kernel function $f(A,s)$ is continuous and differentiable in $s$ for all $s \in \real \wo \{0\}$ and $f(A, 0^-) = \lim_{s \nearrow 0} f(A, s) = C_1 \in M_{n,d}$ as well as $f(A, 0^+) = \lim_{s \searrow 0} f(A, s) = C_2 \in M_{n,d}$ for all $A \in M_d^-$. Set
\[f' (A, s) := \left\{\begin{array}{l@{\quad}l} \frac{d}{ds} f (A, s) &\mbox{if } s \not= 0,\\ \lim_{s \searrow 0} \frac{d}{ds} f (A, s) & \mbox{if } s = 0\end{array}\right.\]
and assume that for some $\delta > 0$ and for every $t_1, t_2 \in \real$ such that $t_1 \leq t_2$ and $t_2 - t_1 \leq \delta$ the function $\sup_{t \in [t_1, t_2]} \| f'(A, t - s)\|$ satisfies the conditions of Proposition \ref{prop:MMAexfinite}, where $(\gamma, 0, \nu, \pi)$ is replaced by $(\|\gamma\|, 0, \nu_T  , \pi)$ and the L\'evy measure $\nu_T  (\cdot) = \nu (T^{-1} (\cdot))$ is transformed by $T(x) = \|x\|$. If the process $X_t = \int_{M_d^-}\int_{\real} f(A,t-s) \Lambda (dA, ds)$ exists (in the sense of Proposition \ref{prop:MMAexfinite}), then setting 
\[Z_t := \int_{M_d^-}\int_{\real} f' (A,t-s) \Lambda (dA, ds)\]
 we have
\begin{equation}X_t = X_0 + \int\limits_0^t Z_u\;du + (C_1 - C_2)\; L_t\label{eq:sde}\end{equation}
and consequently $X_t$ has sample paths in $\dd$ which are of finite variation on compacts.
\end{theorem}

\begin{proof}
Obviously the process $Z_t$ exists (in the sense of Proposition \ref{prop:MMAexfinite}).

We follow the steps of the proof of Theorem 3.12 in \cite{barndorff:stelzer:2011a} and begin by showing that $Z_t$ is locally uniformly bounded on compacts. Note that by Proposition \ref{prop:MMAexfinite} the processes $X_t$ and $Z_t$ can be given as integrals with respect to a Poisson measure and $\pi \times \lambda$. For $\delta > 0$ and every $t_1, t_2 \in \real$ such that $t_1 \leq t_2$ and $t_2 - t_1 \leq \delta$ we obtain
\begin{align*}
\sup\limits_{t \in [t_1, t_2]} \| Z_t \| &= \sup\limits_{t \in [t_1, t_2]} \Big\| \int_{M_d^-}\int_{\real} f' (A,t-s) \Lambda (dA, ds) \Big\| \\
& \leq \int_{M_d^-}\int_{\real}\; \sup\limits_{t \in [t_1, t_2]} \|f' (A,t-s)\| \Lambda_T (dA, ds),
\end{align*}
where $T: \reald \to \real$ is given by $T(x) = \|x\|$ and $\Lambda_T$ is the transformed L\'evy basis with characteristic triplet $(\|\gamma\|, 0, \nu_T  , \pi)$. Existence of the right hand side is covered by Proposition \ref{prop:MMAexfinite}. Thus $Z_t$ is locally uniformly bounded and it follows by Fubini that\enlargethispage{1cm}
\begin{align*}
\int\limits_0^t Z_u du&= \int\limits_0^t \int\limits_{M_d^-}\int\limits_{-\infty}^u f' (A,u-s) \Lambda (dA, ds)\; du + \int\limits_0^t \int\limits_{M_d^-}\int\limits_{u}^\infty f' (A,u-s) \Lambda (dA, ds)\; du\\
&= \int\limits_{M_d^-}\int\limits_{-\infty}^t \int\limits_{0\vee s}^t f' (A,u-s)\; du \, \Lambda (dA, ds)+  \int\limits_{M_d^-}\int\limits_{0}^\infty \int\limits_0^{t \wedge s} f' (A,u-s)\; du \, \Lambda (dA, ds)\\
&= \int\limits_{M_d^-}\int\limits_{-\infty}^t f (A,u-s)\Big|_{u = 0\vee s}^t \, \Lambda (dA, ds)+  \int\limits_{M_d^-}\int\limits_{0}^\infty f (A,u-s)\Big|_{u = 0}^{t \wedge s} \, \Lambda (dA, ds)\\
&= \int\limits_{M_d^-}\int\limits_{-\infty}^t f (A,t-s) \Lambda (dA, ds) - \int\limits_{M_d^-}\int\limits_{-\infty}^0 f (A,0-s) \Lambda (dA, ds)\\
&\quad - \int\limits_{M_d^-}\int\limits_{0}^t f (A,0^+) \Lambda (dA, ds) +  \int\limits_{M_d^-}\int\limits_{t}^\infty f (A,t-s) \Lambda (dA, ds)\\
&\quad + \int\limits_{M_d^-}\int\limits_{0}^t f (A,0^-) \Lambda (dA, ds) - \int\limits_{M_d^-}\int\limits_{0}^\infty f (A,0 - s) \Lambda (dA, ds)\\
&= X_t - X_0 + (C_1 - C_2)\; L_t.
\end{align*}
\end{proof}

\begin{remark}
\begin{enumerate}
 \item The inclusion of kernel functions with a discontinuity at $s = 0$ is motivated by the class of causal MMA processes where the kernel function is of the form $f(A,s) \id_{[0, \infty)} (s)$. For example, in the supOU case the kernel function is $e^{As}\, \id_{[0, \infty)} (s)$ and the limits at $s = 0$ can be given directly by $C_1 = \mathbf{0}$ and $C_2 = I_d$ yielding (see Theorem 3.12 in \cite{barndorff:stelzer:2011a})
\[X_t = X_0 + \int\limits_0^t Z_u\;du - L_t.\]
\item The result can obviously be extended to the case where $f(A,0^-)$ and $f(A,0^+)$ exist for all $A$, but are not independent of $A$. Then \eqref{eq:sde} holds with the L\'evy process $\tilde L_t:=\int\limits_{M_d^-}\int\limits_{0}^t (f (A,0^+)-f (A,0^-)) \Lambda (dA, ds)$ in place of $(C_1 - C_2)\; L_t$.
\end{enumerate}

\end{remark}

If $C_1 - C_2 = 0$ in the above theorem, further properties of the sample paths of $X_t$ follow directly.

\begin{corollary}
Assume that the conditions of Theorem \ref{thm:cadlagmma} hold. If additionally $C_1 = C_2$, then the paths of $X_t = \int_{M_d^-}\int_{\real} f(A,t-s) \Lambda (dA, ds)$ are absolutely continuous and almost everywhere differentiable.
\end{corollary}

\begin{remark}
The condition $C_1 = C_2$ holds if and only if $f(A,s)$ is continuous in $s = 0$ and $f(A,0)$ is constant for all $A \in M_d^-$. This is satisfied, for example, by two-sided supOU processes which are MMA processes with kernel function
\[f(A,s) = e^{As} \id_{[0,\infty)} (s) + e^{-As} \id_{(-\infty, 0)} (s).\]
 In the case of moving average processes, where $\pi$ is a one-point measure, the condition only requires that $f$ is continuous in $s = 0$. Processes of this class include, for example, two-sided CARMA and two-sided Ornstein-Uhlenbeck processes.
\end{remark}

Similar results for the sample paths of MMA processes, where the driving L\'evy process is not of finite variation, are in general not so easy to obtain. \cite{basse:pedersen:2009}, Corollary 3.3, give necessary and sufficient conditions for filtered L\'evy processes of the form \eqref{eqn:filteredlevy} to have c\`adl\`ag sample paths of bounded variation even if the driving L\'evy process itself has sample paths of unbounded variation. Furthermore, they also study \notiono{two-sided moving averages} of the form
\[X_t = \int\limits_{-\infty}^t (f_1 (t - s) - f_2 (-s))\; d L_t,\]
where $f_1, f_2: \real \to \real$ are measurable kernel functions such that $f_1 (s) = f_2 (s) = 0$ for all $s \in (-\infty, 0)$. They give necessary and sufficient conditions for such processes to have c\`adl\`ag sample paths of finite variation. These conditions also allow the underlying L\'evy process to be of infinite variation. For MMA processes \cite{BasseRosinski2012} gives necessary and sufficient conditions for finite variation and absolute continuity of sample paths for underlying L\'evy processes of infinite variation. Moreover, \cite{basse:pedersen:2009} also consider the special case where the driving L\'evy process is symmetric $\alpha$-stable with $\alpha \in (1,2]$ (cf.\ \cite{basse:pedersen:2009}, Lemma 5.2, Proposition 5.3 and Proposition 5.5). Conditions for $\alpha$-stable MMA processes, $\alpha \in (0,2)$, to have c\`adl\`ag sample paths are also given in \cite{basse:rosinski:2011}, Section 4.\par
Additionally, there also exist some results for the stronger property of continuous sample paths. See \cite{marcus:rosinski:2005}, \cite{cambanis:nolan:rosinski:1990} and \cite{rosinski:1989} for results on general MMA processes to have continuous sample paths. For the special case of $\alpha$-stable MMA processes, see also \cite{rosinski:samorodnitsky:taqqu:1991} and \cite{rosinski:1986}.

\section{Functional Regular Variation}\label{sec:dd:funcregvar}

We follow \cite{hult:lindskog:2005} to introduce the notion of regular variation on $\dd$. Let $\dd$ be the space of c\`adl\`ag\index{C\`adl\`ag functions} (right-continuous with left limits) functions $x : [0,1] \to \reald$ equipped with the $J_1$ metric (equivalent to the $d_0$ metric of \cite{billingsley:1968}) such that $\dd$ is a complete and separable metric space. Using the supremum norm\index{Supremum norm} $\|x\|_\infty = \sup_{t \in [0,1]} \|x_t\|$ we can then introduce the unit sphere $\sdd = \{x \in \dd: \|x\|_\infty = 1\}$, equipped with the relativized topology of $\dd$. Next, we equip $(0, \infty ]$ with the metric $\rho (x,y) = |1/x - 1/y|$ which makes it a complete separable metric space. Then also the space $\ddcon = (0, \infty ] \times \sdd$, equipped with the metric $\max \{\rho (x^*,y^*), d_0(\schl{x}, \schl{y})\}$, is a complete separable metric space.\par
If we use the polar coordinate transformation $T: \dd \wo \{0\} \to \ddcon,$ $x \mapsto (\|x\|_\infty, x / \|x\|_\infty ),$ we see that the spaces $\dd \wo \{0\}$ and $(0, \infty ) \times \sdd$ are homeomorphic. Thus, the Borel sets $\borel (\ddcon )$ of interest can be viewed as the infinite dimensional extension of the one-point uncompactification\index{One-point uncompactification} that is used to introduce finite dimensional regular variation (cf.\ \cite{bingham:goldie:teugels:1987}, \cite{embrechts:klueppelberg:mikosch:1997} and \cite{resnick:1987}).\par
Regular Variation on $\dd$ can then be introduced in terms of the so-called $\what$-convergence of boundedly finite measures on $\ddcon$. A measure $\mu$ on a complete separable metric space $E$ is said to be \notionOI{boundedly finite}\index{Boundedly finite} if $\mu(B) < \infty$ for every bounded set $B \in \borel (E)$. Let $(\mu_n)_{n \in \nat}$ be a sequence of boundedly finite measures on $E$. Then $(\mu_n)$ converges to $\mu$ in the $\what$\notionOI{-topology}\index{Weak hat topology} if $\mu_n(B) \to \mu(B)$ for all bounded Borel sets $B \in \borel (E )$ with $\mu(\partial B) = 0$. We write $\mu_n \tow \mu$. Note that for locally compact spaces $E$ the boundedly finite measures are called Radon measures and the notions of $\what$-convergence and vague convergence coincide. See \cite{daley:verejones:1988} and \cite{kallenberg:1983} for details on $\what$-convergence and vague convergence.\par
We are now able to formulate regular variation for stochastic processes with sample paths in $\dd$.

\begin{definition}[Regular Variation on $\dd$]\label{def:dd:regvarD}
A stochastic process $(X_t)$, $t \in [0,1]$, with sample paths in $\dd$ is said to be \notionOI{regularly varying} if there exists a positive sequence $(a_n)$, $n \in \nat$, with $a_n \nearrow \infty$ and a nonzero boundedly finite measure $\mu$ on $\borel (\ddcon)$ with $\mu(\ddcon \wo \dd) = 0$ such that, as $n \to \infty$,
\[n P (a_n^{-1} X \in \cdot) \tow \mu( \cdot ) \quad \mbox{ on } \borel (\ddcon).\]
\end{definition}

As in the finite dimensional case, direct calculation shows that the measure $\mu$ is homogeneous, i.e.\ there exists a positive index $\alpha > 0$ such that $\mu(u B) = u^{-\alpha} \mu(B)$ for all $u > 0$ and for every $B \in \borel (\ddcon)$. Thus, we say that the process $(X_t)$ is \notiono{regularly varying with index} $\alpha$ and write $X \in \rvdd$.\par
Analogously to multivariate regular variation, several alternative definitions of regular variation on $\dd$ exist, we will just state one example here.

\begin{theorem}[\cite{hult:lindskog:2005}, Theorem 4]
A process $(X_t)$ with sample paths in $\dd$ is regularly varying if and only if there exists an index $\alpha > 0$ and a probability measure $\sigma$ on $\borel (\sdd)$ such that for every positive $x > 0$, as $u \to \infty$, 
\[\frac{P( \|X\|_\infty > ux, X / \|X\|_\infty \in \cdot)}{P(\|X\|_\infty > u)} \tow x^{-\alpha} \sigma (\cdot) \quad \mbox{ on } \borel (\sdd).\]
\end{theorem}
The probability measure $\sigma$ is called the \notionOI{spectral measure} of $X$. \par

\begin{example}[L\'evy Process]\label{ex:regvarLevyD}
Let $(L_t)$ be a L\'evy process. Then by definition (or Theorem 11.5 in \cite{sato:2002} resp.) $(L_t)$ has sample paths in $\dd$. Furthermore, $(L_t)$ is also a strong Markov process (cf.\ \cite{sato:2002}, Theorem 10.5 and Corollary 40.11). Now the results of \cite{hult:lindskog:2005}, Section 3, can be applied. If $L_t \in \rvan[t \mu]$ for one and thus all $t>0$, then it follows by Theorem 13 of \cite{hult:lindskog:2005} that $(L_t) \in \rvdd[\schl{\mu}]$ for some measure $\schl{\mu}$. For details we refer to \cite{hult:lindskog:2005}, Example 17.
\end{example}

We will now recall some useful results from \cite{hult:lindskog:2005} related to regular variation on $\dd$. Since it is often of interest, how the regular variation property is preserved under mappings, we look at a continuous mapping theorem. Therefore, for any function $h$ from a metric space $E$ to a metric space $E'$ we introduce the set $disc(h)$ which consists of all discontinuities of $h$.

\begin{theorem}[\cite{hult:lindskog:2005}, Theorem 6]\label{thm:dd:contmapD}
Let $(X_t)$ be a stochastic process with sample paths in $\dd$ and let $E'$ be a complete separable metric space. Assume that $X \in \rvdd$ and $h: \dd \to E'$ is a measurable mapping such that $\mu (disc(h)) = 0$ and $h^{-1} (B)$ is bounded in $\ddcon$ for every bounded $B \in \borel (E')$. Then, as $n \to \infty$,
\[n P (h(a_n^{-1} X) \in \cdot) \tow \mu \circ h^{-1} (\cdot) \quad \mbox{on } \borel (E').\]
\end{theorem}

There also exists a different version of the previous theorem for the special case of positively homogeneous mappings of order $\gamma >0$, i.e.\ mappings $h: \dd \to \dd$ with $h(\lambda x) = \lambda^\gamma h(x)$ for all $\lambda \geq 0$ and $x \in \dd$. See \cite{hult:lindskog:2005}, Theorem 8, for details.\par


The next theorem states some necessary and sufficient conditions for regular variation on $\dd$. In the theorem, we use the notation
\begin{align*}
w(x, T_0) &:= \sup\limits_{t_1, t_2 \in T_0} \; \|x_{t_1} - x_{t_2}\|\quad\mbox{ and}\\ 
w''(x, \delta) &:= \sup\limits_{0\leq t_1 \leq t \leq t_2\leq 1;\; t_2 - t_1 \leq \delta}\; \min \; \{ \|x_{t} - x_{t_1}\|, \|x_{t_2} - x_{t}\| \}
\end{align*}
for $x \in \dd$, $T_0 \subseteq [0,1]$ and $\delta \in [0,1]$.

\begin{theorem}[\cite{hult:lindskog:2005}, Theorem 10]\label{thm:dd:charregvarD}
Let $(X_t)$ be a stochastic process with sample paths in $\dd$. Then the following statements are equivalent.
\begin{list}{}{\setlength{\leftmargin}{0.8cm}\setlength{\labelwidth}{0.7cm}\setlength{\labelsep}{0.1cm}}
	\item[(i)] $X \in \rvdd$.
	\item[(ii)] There exists a set $T \subseteq [0,1]$ containing $0$, $1$ and all but at most countably many points of $[0,1]$, a positive sequence $a_n \nearrow \infty$ and a collection $\{ \mu_{t_1, \ldots, t_k} : t_i \in T, k \in \nat \}$ of Radon measures on $\borel (\bar{\real}^{dk} \wo \{0\})$ with $\mu_{t_1, \ldots, t_k} (\bar{\real}^{dk} \wo \real^{dk}) = 0$ and $\mu_t$ is nonzero for some $t\in T$ such that
\begin{equation}
n P (a_n^{-1} ( X_{t_1}, \ldots, X_{t_k} ) \in \cdot) \tov \mu_{t_1, \ldots, t_k} (\cdot)  \quad \mbox{on } \borel (\bar{\real}^{dk} \wo \{0\})
\end{equation}
holds for all $t_1, \ldots, t_k \in T$. Furthermore, for every $\varepsilon, \eta > 0$, there exist $\delta \in (0,1)$ and $n_0 \in \nat$ such that, for all $n \geq n_0$,
\begin{align}
n P ( a_n^{-1} w(X, [0, \delta)) \geq \varepsilon) &\leq \eta \label{eqn:dd:relcomp1},\\
n P ( a_n^{-1} w(X, [1 - \delta, 1)) \geq \varepsilon) &\leq \eta \label{eqn:dd:relcomp2},\\
n P ( a_n^{-1} w''(X, \delta) \geq \varepsilon) &\leq \eta. \label{eqn:dd:relcomp3}
\end{align}
\end{list}
\end{theorem}

\begin{remark}
The theorem links regular variation of the process $(X_t)_{t\in [0,1]}$ with sample paths in $\dd$ to regular variation of the finite dimensional distributions $( X_{t_1}, \ldots, X_{t_k} )$ of the the process. Key to that connection are the relative compactness criteria \eqref{eqn:dd:relcomp1}, \eqref{eqn:dd:relcomp2} and \eqref{eqn:dd:relcomp3} which restrict the oscillation of the process $(X_t)$ in small areas. See \cite{hult:lindskog:2005}, Example 11, for a process satisfying conditions \eqref{eqn:dd:relcomp1} and \eqref{eqn:dd:relcomp2}, but not \eqref{eqn:dd:relcomp3}.
\end{remark}

Now we will extend the finite dimensional regular variation of MMA processes in the sense of Theorem \ref{thm:dd:regvarMMA} to regular variation in $\dd$ by applying Theorem \ref{thm:dd:charregvarD}. Therefore, we need to restrict the MMA process $(X_t)$ as defined in Section \ref{sec:dd:prelim:mma} to the time interval $[0,1]$. Note that a restriction to any other compact time interval $[a,b]$, $a < b$, would not change any of the results. Furthermore, we assume that $(X_t)$ has sample paths in the space $\dd$ of c\`adl\`ag functions. See Section \ref{sec:dd:sample} for possible conditions ensuring this. We start with the main theorem for functional regular variation of MMA processes.

In order not to overload the notation we from now on assume always that $t,t_1, t_2$ are restricted to the set $[0;1]$ when taking suprema without writing this explicitly. Furthermore, we are now using the decomposition \eqref{eq:lam2} of our L\'evy basis $\Lambda$ into $\Lambda^{(1)}$ and $\Lambda^{(2)}$.

\begin{theorem}\label{thm:funcregvarmma}
Let $\Lambda$  be an $\reald$-valued L\'evy bases on $M_d^- \times \real$ with generating quadruple $(\gamma, \Sigma, \nu, \pi)$  such that $\nu \in \rvan[\mu_\nu]$. Assume that the kernel function $f$ is bounded, $f \in \mathbb{L}^\alpha (\lambda \times \pi)$, $\mu_\nu ( f^{-1} (A,s) (\real^n \wo \{0\}) ) = 0$ does not hold for $\pi \times \lambda$ almost-every $(A,s)$ and
\begin{equation}\label{eqn:xt2ex}
\int \limits_{M_d^-} \int \limits_{\real} \int \limits_{\|x\| > 1} (1 \wedge \| f(A, s) x\|) \,\nu (dx)\,ds\,\pi (dA) < \infty.
\end{equation}
Moreover, suppose that the MMA process $X_t = \int_{M_d^-}\int_{\real} f(A,t-s) \Lambda (dA, ds)$ exists for $t \in [0,1]$ (in the sense of Theorem \ref{thm:dd:MMAex}) and that the processes $X_t$ and $X_t^{(2)} = \int_{M_d^-}\int_{\real} f(A,t-s) \Lambda^{(2)} (dA, ds)$ have c\`adl\`ag sample paths. If the function $f_\delta$ given by
\begin{equation}\label{eqn:fdelta}
f_{\delta} (A,s) := \sup\limits_{t_1 \leq t_2;\; t_2 - t_1 \leq \delta} \| f(A, t_2 - s) - f(A, t_1 - s) \|\; \id_{(t_1,t_2]^c}\, (s)
\end{equation}
satisfies \eqref{eqn:xt2ex} and, as $\delta \to 0$, 
\begin{equation}\label{eqn:funcregvarmmabed}
\int\limits_{M_d^-} \int\limits_{\real} f_{\delta} (A,s)^\alpha\, d s \pi (d A) \to 0,
\end{equation}
then 
\[(X_t)_{t \in [0,1]} \in \rvdd ,\]
where $\mu$ is uniquely determined by the measures $\mu_{t_1, \ldots, t_k}$ in Theorem \ref{thm:dd:regvarMMA}.
\end{theorem}


%
 Note that an essential part of the proof of the  theorem is going to be to show that under the above assumption it is essentially only $X^{(2)}$ that matters regarding the extremal behaviour.

The condition $\int_{M_d^-} \int_{\real} f_\delta (A,s)^\alpha\, d s \pi (d A) \to 0$ is closely linked to the behavior of the function $f_\delta$ over small intervals. It restricts the amplitude of jumps and continuous oscillations for arbitrarily small values of $\delta$.

\begin{lemma}\label{lemma:kvgzpunkwint}
Let $\pi$ be a probability measure and $f : M_d^- \times \real \mapsto M_{n,d}$ be a measurable kernel function. Assume that the function $f_{\delta} (A,s)$ given by \eqref{eqn:fdelta} satisfies $f_\delta \in \mathbb{L}^\alpha$ for some $\delta > 0$. Then
\[\lim\limits_{\delta \to 0}\;\int\limits_{M_d^-} \int\limits_{\real} f_\delta (A,s)^\alpha\, d s \pi (d A) = 0\]
if and only if $\;\lim_{\delta \to 0} \; f_\delta (A,s) \to 0$ for $\pi \times \lambda$ almost every $(A,s)$. 
\end{lemma}

\begin{proof}
Let $f_\delta (A,s) \to 0$ for $\pi \times \lambda$ almost every $(A,s)$. Then
\[\int\limits_{M_d^-} \int\limits_{\real} f_\delta (A,s)^\alpha\, d s \pi (d A) \to 0\]
follows by dominated convergence and the assumption $f_\delta \in \mathbb{L}^\alpha$ for some $\delta > 0$. On the other hand, suppose that the set
\[\tilde{B} := \big\{ (A,s) \in \borel_b (M_d^- \times \real) : f_\delta (A,s) \to 0 \mbox{ as } \delta \to 0\big\}\]
satisfies $\pi \times \lambda\; (\tilde{B}^c) = C > 0$. Then the monotonicity of $f_\delta$ in $\delta$ implies $\lim_{\delta \to 0} f_\delta (A,s) > 0$ for every $(A,s) \in \tilde{B}^c$ and thus
\[\lim\limits_{\delta \to 0}\;\int\limits_{M_d^-} \int\limits_{\real} f_\delta (A,s)^\alpha\, d s \pi (d A) > 0.\]
\end{proof}

\begin{remark}
From the definition of $f_\delta$ we see that the condition $\;\lim_{\delta \to 0} \; f_\delta (A,s) \to 0$ for $\pi \times \lambda$ almost every $(A,s)$ is equivalent to the kernel function $f(A,s)$ being continuous in $s$ for all $s \in \real \wo \{0\}$. Now we also see the importance of the restriction $\id_{(t_1,t_2]^c}\, (s)$ in the definition of $f_\delta$ because it allows for $f(A,s)$ being discontinuous at $s = 0$. Without such a restriction, condition \eqref{eqn:funcregvarmmabed} would be violated by many examples of the class of causal MMA processes which have a kernel function of the type $f(A,s) \id_{[0, \infty)} (s)$. Causal MMA processes with $f(A, 0) \not= 0$ include CARMA and supOU processes as well as many other well-known examples of MMA processes.
\end{remark}

We can also give sufficient conditions for a general function $f: M_d^- \times \real \to M_{n,d}$ to satisfy condition \eqref{eqn:xt2ex} which is going to ensure that $X_t^{(2)}$ is well-defined as an $\omega$-wise Lebesgue integral.

\begin{lemma}\label{lemma:charxt2ex}
Let $f: M_d^- \times \real \to M_{n,d}$ be a measurable function. Then condition \eqref{eqn:xt2ex} holds if one of the following two conditions are satisfied:
\begin{enumerate}
	\item[(i)] $f \in \mathbb{L}^1$ and $\alpha > 1$.
	\item[(ii)] $f \in \mathbb{L}^{\alpha - \varepsilon}$ for one $\varepsilon \in (0,\alpha)$ and $\alpha \leq 1$.  
\end{enumerate}
\end{lemma}

\begin{proof}
For (i) we calculate
\begin{align*}
&\int \limits_{M_d^-} \int \limits_{\real} \int \limits_{\|x\| > 1} (1 \wedge \| f(A, s) x\|) \,\nu (dx)\,ds\,\pi (dA) \leq\\
&\achtquad \leq \int \limits_{M_d^-} \int \limits_{\real} \int \limits_{\|x\| > 1}  \| f(A, s)\| \|x\| \,\nu (dx)\,ds\,\pi (dA)\\
&\achtquad = \int \limits_{M_d^-} \int \limits_{\real}  \| f(A, s)\| \,ds\,\pi (dA)\; \int \limits_{\|x\| > 1}  \| x\| \,\nu (dx)\\
&\achtquad < \infty
\end{align*}
by \cite{sato:2002}, Corollary 25.8, and similarly for (ii) we obtain
\begin{align*}
&\int \limits_{M_d^-} \int \limits_{\real} \int \limits_{\|x\| > 1} (1 \wedge \| f(A, s) x\|) \,\nu (dx)\,ds\,\pi (dA) \leq\\
&\achtquad \leq \int \limits_{M_d^-} \int \limits_{\real} \int \limits_{\|x\| > 1} (1 \wedge \| f(A, s) x\|^{\alpha - \varepsilon}) \,\nu (dx)\,ds\,\pi (dA)\\
&\achtquad \leq \int \limits_{M_d^-} \int \limits_{\real} \int \limits_{\|x\| > 1}  \| f(A, s)\|^{\alpha - \varepsilon} \|x\|^{\alpha - \varepsilon} \,\nu (dx)\,ds\,\pi (dA)\\
&\achtquad = \int \limits_{M_d^-} \int \limits_{\real}  \| f(A, s)\|^{\alpha - \varepsilon} \,ds\,\pi (dA)\; \int \limits_{\|x\| > 1}  \| x\|^{\alpha - \varepsilon} \,\nu (dx)\\
&\achtquad < \infty.
\end{align*}
\end{proof}

\begin{remark}
The conditions of Lemma \ref{lemma:charxt2ex} are only sufficient, not necessary, similar to the ones of Theorem \ref{thm:dd:MMAex2}. Thus in general we will only demand the weaker condition \eqref{eqn:xt2ex} which is also one of the existence conditions for MMA processes with driving L\'evy process of finite variation in Proposition \ref{prop:MMAexfinite}. Furthermore, we see from Lemma \ref{lemma:charxt2ex}(i) that condition \eqref{eqn:xt2ex} in Proposition \ref{prop:MMAexfinite} can be dropped if $\alpha > 1$.
\end{remark}\vspace*{0.2cm}

\subsection{Proof of Theorem \ref{thm:funcregvarmma}}~\\

In this subsection we now gradually prove Theorem \ref{thm:funcregvarmma}.

Let $(X_t)$ be an MMA process as given in Theorem \ref{thm:funcregvarmma}, i.e.\ $(X_t)$ exists for $t \in [0,1]$ (in the sense of Theorem \ref{thm:dd:MMAex}), the kernel function $f$ is bounded by $C \in \real^+$ and the regular variation conditions of Theorem \ref{thm:dd:regvarMMA} hold. Then there exists a positive sequence $a_n \nearrow \infty$ and a collection $\{ \mu_{t_1, \ldots, t_k} : t_i \in T, k \in \nat \}$ of Radon measures on $\borel (\bar{\real}^{dk} \wo \{0\})$ with $\mu_{t_1, \ldots, t_k} (\bar{\real}^{dk} \wo \real^{dk}) = 0$ and $\mu_t$ is nonzero for some $t\in T$ such that
\begin{equation*}
n P (a_n^{-1} ( X_{t_1}, \ldots, X_{t_k} ) \in \cdot) \tov \mu_{t_1, \ldots, t_k} (\cdot)  \quad \mbox{on } \borel (\bar{\real}^{dk} \wo \{0\}).
\end{equation*}
Applying Theorem \ref{thm:dd:charregvarD}, it is left to show that the conditions \eqref{eqn:dd:relcomp1}, \eqref{eqn:dd:relcomp2} and \eqref{eqn:dd:relcomp3} hold.\par
Using the L\'evy-It\^o decomposition we have two independent L\'evy bases $\Lambda^{(1)}$ and $\Lambda^{(2)}$ such that $\Lambda^{(1)}$ has generating quadruple $(\gamma,\Sigma, \nu_1, \pi)$ and $\Lambda^{(2)}$ has generating quadruple $(0, 0, \nu_2 , \pi)$, where $\nu_1 = \nu |_{B_1(0)}$ and $\nu_2 = \nu |_{B_1(0)^c}$. This yields
\begin{equation}
X_{t} = X_{t}^{(1)} + X_{t}^{(2)},
\end{equation}\label{eqn:levyitomma}
where
\begin{equation}\label{eqn:levyitommabound}
X_{t}^{(1)} = \int\limits_{M_d^-}\int\limits_{\real} f (A,t-s) \Lambda^{(1)}(dA, ds)
\end{equation}
and
\begin{equation}\label{eqn:levyitommafinite}
X_{t}^{(2)} = \int\limits_{M_d^-}\int\limits_{\real} f (A,t-s) \Lambda^{(2)} (dA, ds).
\end{equation}
Note that the term $X_t^{(2)}$ can be written in the form
\[X_t^{(2)} = \int\limits_{\|x\| \geq 1}\int\limits_{M_d^-}\int\limits_{\real} f (A,t - s) \,x\, N(dx, dA, ds),\]
where $N$ is a Poisson random measure with mean measure $\nu \times \pi \times \lambda$. Before we proceed, we need to ensure the existence of $X_t^{(1)}$ and $X_t^{(2)}$. Therefore, we give conditions for $\omega$-wise existence of $X_t^{(2)}$ as a Lebesgue integral. Then the existence of $X_t^{(1)} = X_t - X_t^{(2)}$ follows from the existence of $X_t$ and $X_t^{(2)}$.

\begin{proposition}\label{prop:xt2exist}
Let $X_t^{(2)}$ be the process given by \eqref{eqn:levyitommafinite}, where $\Lambda^{(2)}$ is a L\'evy basis with generating quadruple $(0, 0, \nu_2 , \pi)$. If
\[\int \limits_{M_d^-} \int \limits_{\real} \int \limits_{\reald} (1 \wedge \| f(A, s) x\|) \,\nu_2 (dx)\,ds\,\pi (dA) < \infty,\]
then $X_t^{(2)}$ exists as a Lebesgue integral for all $\omega \in \Omega$.
\end{proposition}

\begin{proof}
By definition, $X_t^{(2)}$ has no Gaussian component and $\int_{\|x\|\leq 1} \|x\| \nu_2 (dx) = 0$ and thus we have an underlying L\'evy process of finite variation. Now the result follows as a special case of Proposition \ref{prop:MMAexfinite}, where the condition $f \in \mathbb{L}^1$ is obsolete due to the absence of a drift.
\end{proof}

Like for $X_t$, we also assumed that $X_t^{(2)}$ has c\`adl\`ag sample paths. Then also $X_t^{(1)} = X_t - X_t^{(2)}$ has c\`adl\`ag sample paths. Appropriate conditions for MMA processes to have c\`adl\`ag sample paths have been given in Section \ref{sec:dd:sample}.\par
Now we continue the proof of Theorem \ref{thm:funcregvarmma} by verifying the relative compactness conditions \eqref{eqn:dd:relcomp1}, \eqref{eqn:dd:relcomp2} and \eqref{eqn:dd:relcomp3}. The essential point is to relate the conditions back to the analogous conditions on the underlying L\'evy process.

For the first condition \eqref{eqn:dd:relcomp1} we obtain
\begin{align*}
\sup\limits_{t_1, t_2 \in [0,\delta)} \; \|X_{t_1} - X_{t_2}\|  &\leq \sup\limits_{t_1, t_2 \in [0,\delta)} \; \|X_{t_1}^{(1)} - X_{t_2}^{(1)}\| + \sup\limits_{t_1, t_2 \in [0,\delta)} \; \|X_{t_1}^{(2)} - X_{t_2}^{(2)}\|
\end{align*}
and hence
\begin{align*}
&n P \Big( a_n^{-1} \sup\limits_{t_1, t_2 \in [0,\delta)} \; \|X_{t_1} - X_{t_2}\| \geq \varepsilon \Big) \leq\notag\\
&\quad\leq n P \Big( a_n^{-1} \sup\limits_{t_1, t_2 \in [0,\delta)}  \|X_{t_1}^{(1)} - X_{t_2}^{(1)}\| \geq \varepsilon / 2 \Big) + n P \Big( a_n^{-1} \sup\limits_{t_1, t_2 \in [0,\delta)}  \|X_{t_1}^{(2)} - X_{t_2}^{(2)}\| \geq \varepsilon / 2 \Big).
\end{align*}
The analogue result for the second condition \eqref{eqn:dd:relcomp2} can be obtained likewise. For the third condition \eqref{eqn:dd:relcomp3} we estimate
\begin{align*}
&\sup\limits_{t_1\leq t \leq t_2;\; t_2 - t_1 \leq \delta} \min \Big\{ \|X_{t_2} - X_{t}\|, \|X_{t} - X_{t_1}\|\Big\} \leq\\
&\zweiquad \leq \sup\limits_{t_1 \leq t_2;\; t_2 - t_1 \leq \delta}\; \|X_{t_1}^{(1)} - X_{t_2}^{(1)}\| + \sup\limits_{t_1\leq t \leq t_2;\; t_2 - t_1 \leq \delta} \min \Big\{ \|X_{t_2}^{(2)} - X_{t}^{(2)}\|, \|X_{t}^{(2)} - X_{t_1}^{(2)}\|\Big\},
\end{align*}
and
\begin{align*}
&n P \Big( a_n^{-1} \sup\limits_{t_1\leq t \leq t_2;\; t_2 - t_1 \leq \delta} \min \Big\{ \|X_{t_2} - X_{t}\|, \|X_{t} - X_{t_1}\|\Big\} \geq \varepsilon \Big) \leq\notag\\
&\vierquad\zweiquad\quad \leq n P \Big( a_n^{-1} \sup\limits_{t_1 \leq t_2;\; t_2 - t_1 \leq \delta}\; \|X_{t_1}^{(1)} - X_{t_2}^{(1)}\| \geq \varepsilon / 2 \Big)\\
&\achtquad + n P \Big( a_n^{-1} \sup\limits_{t_1\leq t \leq t_2;\; t_2 - t_1 \leq \delta} \min \Big\{ \|X_{t_2}^{(2)} - X_{t}^{(2)}\|, \|X_{t}^{(2)} - X_{t_1}^{(2)}\|\Big\} \geq \varepsilon / 2 \Big).
\end{align*}
For every $\varepsilon, \eta > 0$ we have to show that there exists $n_0 \in \nat$ and $\delta > 0$ such that for $n \geq n_0$ these quantities can be bounded by $\eta$. Regarding the quantities based on $X_t^{(1)}$ we observe
\[n P \Big( a_n^{-1} \sup\limits_{t_1, t_2 \in [0,\delta)}  \|X_{t_1}^{(1)} - X_{t_2}^{(1)}\| \geq \varepsilon / 2 \Big) \leq n P \Big( a_n^{-1} \sup\limits_{t_1 \leq t_2;\; t_2 - t_1 \leq \delta}\; \|X_{t_1}^{(1)} - X_{t_2}^{(1)}\| \geq \varepsilon / 2 \Big)\]
and for \eqref{eqn:dd:relcomp2}
\[n P \Big( a_n^{-1} \sup\limits_{t_1, t_2 \in [1 - \delta,1)}  \|X_{t_1}^{(1)} - X_{t_2}^{(1)}\| \geq \varepsilon / 2 \Big) \leq n P \Big( a_n^{-1} \sup\limits_{t_1 \leq t_2;\; t_2 - t_1 \leq \delta}\; \|X_{t_1}^{(1)} - X_{t_2}^{(1)}\| \geq \varepsilon / 2 \Big)\]
and thus it is sufficient to prove the bound only for the right hand side of the inequality.

\begin{proposition}\label{prop:funcregvarmma1}
Let $\Lambda^{(1)}$ be the $\reald$-valued L\'evy basis on $M_d^- \times \real$ determined by the generating quadruple $(\gamma, \Sigma, \nu_1, \pi)$, where $\nu_1 = \nu |_{B_1(0)}$. Assume that the kernel function $f$ is bounded, that the MMA process $X_t^{(1)}$ given by \eqref{eqn:levyitommabound} exists for $t \in [0,1]$ and that $X_t^{(1)}$ has c\`adl\`ag sample paths. Moreover, suppose that $\nu \in \rvan[\mu_\nu]$. Then $X_t^{(1)}$ satisfies 
\[\lim\limits_{n \to \infty} n P ( a_n^{-1} \sup\limits_{t_1 \leq t_2;\; t_2 - t_1 \leq \delta} \; \|X_{t_1}^{(1)} - X_{t_2}^{(1)}\| \geq \varepsilon) = 0\]
for all $\delta \in (0,1)$ and $\varepsilon > 0$.
\end{proposition}
\begin{proof}
We start by observing that $X_{t}^{(1)}$ is c\`adl\`ag and thus also separable and hence we can estimate
\[\sup\limits_{t_1 \leq t_2;\; t_2 - t_1 \leq \delta} \; \|X_{t_1}^{(1)} - X_{t_2}^{(1)}\| \leq 2 \sup\limits_{t \in [0,1]} \; \|X_{t}^{(1)}\| = 2 \sup\limits_{t \in [0,1] \cap \rational} \; \|X_{t}^{(1)}\|.\]
Due to the equivalence of norms, we can now choose the matrix norm
\[\|A \| := \max \{|a_{ij}|: 1 \leq i \leq n \mbox{ and } 1 \leq j \leq d\}\]
for $A \in M_{n,d}$ and denote by $X_{t, i}^{(1)} \in \real$, $1 \leq i \leq n$, the $i$-th component of $X_{t}^{(1)}$, i.e.
\[X_{t}^{(1)} = \left( X_{t,1}^{(1)}, X_{t,2}^{(1)}, \ldots , X_{t,n}^{(1)} \right)^T.\]
Furthermore, define the (countable) set
\[\schl{T} := \left\{ (t, i): t \in [0,1] \cap \rational \mbox{ and } i \in \{1, \ldots, n\} \right\}.\]
Then we obtain
\[ \sup\limits_{t \in [0,1] \cap \rational} \; \|X_{t}^{(1)}\| = \sup\limits_{t \in [0,1] \cap \rational} \;\max\limits_{1 \leq i \leq n} \|X_{t,i}^{(1)}\| = \sup\limits_{s \in \schl{T}} \; \|X_{s}^{(1)}\|,\]
where $\sup_{s \in \schl{T}}$ is a subadditive functional on $\real^{\schl{T}}$. Furthermore, by Theorem \ref{thm:dd:MMAex} the processes $X_{t,i}^{(1)}$ are infinitely divisible with specified characteristic triplet $(\gamma_{t,i}, \Sigma_{t,i}, \nu_{t,i})$ and L\'evy measure
\[\nu_{t,i} (B)= \int\limits_{M_d^-} \int\limits_{\real} \int\limits_{\reald} \id_B(f_i (A,t - s)x) \nu_1(d x) d s \pi (d A)\]
for all $B \in \borel (\real)$, where $f_i$ denotes the $i$-th row of $f$, i.e.
\[
 f(A,t-s)=\begin{pmatrix} 
           f_1(A,t-s)\\
\vdots\\
f_d(A,t-s)
          \end{pmatrix}.
\]
It follows that $\mathbf{X^{(1)}} = \{X_s^{(1)}: x \in \schl{T} \}$ is infinitely divisible with characteristic triplet $(\schl{\gamma}, \schl{\Sigma}, \schl{\nu})$, where $\schl{\gamma}$, $\schl{\Sigma}$ and $\schl{\nu}$ are given as projective limits of the corresponding finite dimensional characteristics described by $(\gamma_{t,i}, \Sigma_{t,i}, \nu_{t,i})$ (cf.\ \cite{maruyama:1970}). Moreover, the boundedness $\|f\| \leq C$ implies $\|f_i\| \leq C$ and this, together with the definition of $\nu_1 = \nu |_{B_1(0)}$, yields that the support of the L\'evy measures $\nu_{t,i}$ and $\schl{\nu}$ can be bounded by $C$. Now we are able to apply Lemma 2.1 of \cite{braverman:samorodnitsky:1995} to obtain
\[\ev \left( \exp \left( \varepsilon\; \sup\limits_{s \in \schl{T}} \; \|X_{s}^{(1)}\|\right) \right) < \infty\]
for all $\varepsilon > 0$. Finally, the finite exponential moments in combination with Lemma 1.32 of \cite{lindskog:2004} yield
\[\lim\limits_{n \to \infty} n P \Big( a_n^{-1} \sup\limits_{t_1 \leq t_2;\; t_2 - t_1 \leq \delta} \; \|X_{t_1}^{(1)} - X_{t_2}^{(1)}\| \geq \varepsilon \Big) \leq \lim\limits_{n \to \infty} n P \Big( a_n^{-1} \sup\limits_{s \in \schl{T}} \; \|X_{s}^{(1)}\| \geq \varepsilon / 2\Big) = 0\]
for all $\varepsilon > 0$.
\end{proof}

Next we check the process $X_t^{(2)}$ with respect to the relative compactness conditions \eqref{eqn:dd:relcomp1}, \eqref{eqn:dd:relcomp2} and \eqref{eqn:dd:relcomp3}.

\begin{proposition}\label{prop:funcregvarmma2}
Let $\Lambda$ be an $\reald$-valued L\'evy basis on $M_d^- \times \real$ with generating quadruple $(\gamma, \Sigma, \nu, \pi)$ and let $\nu \in \rvan[\mu_\nu]$. Assume that the kernel function $f$ is bounded, the MMA process $X_t^{(2)} = \int_{M_d^-}\int_{\real} f(A,t-s) \Lambda^{(2)} (dA, ds)$ satisfies the existence conditions of Proposition \ref{prop:xt2exist} and that the regular variation conditions of Theorem \ref{thm:dd:regvarMMA} hold. If the function $f_\delta$ given by \eqref{eqn:fdelta}
satisfies the existence condition of Proposition \ref{prop:xt2exist} and, as $\delta \to 0$, 
\[\int\limits_{M_d^-} \int\limits_{\real} f_{\delta} (A,s)^\alpha\, d s \pi (d A) \to 0,\]
then $X_t^{(2)}$ given by \eqref{eqn:levyitommafinite} satisfies the relative compactness conditions \eqref{eqn:dd:relcomp1}, \eqref{eqn:dd:relcomp2} and \eqref{eqn:dd:relcomp3}.
\end{proposition}

\begin{proof}
We define the difference function $g_{t_1, t_2} (A,s) := f(A, t_1 - s) - f(A, t_2 - s)$ and mention that for every $t_1, t_2 \in [0,1]$ the random vector
\[X_{t_1}^{(2)} - X_{t_2}^{(2)} = \int\limits_{\|x\| \geq 1}\int\limits_{M_d^-}\int\limits_{\real} g_{t_1, t_2} (A,s) \;x\; N (dx, dA, ds)\]
is again MMA and by Theorem \ref{thm:dd:MMAex2} and Theorem \ref{thm:dd:regvarMMA} it exists and is regularly varying with index $\alpha$. 
\par\emph{Condition \eqref{eqn:dd:relcomp1}}: We verify the condition by showing that, as $\delta \to 0$,
\[\lim\limits_{n \to \infty}\;n P ( a_n^{-1} \sup\limits_{t_1, t_2 \in [0,\delta)} \; \|X_{t_1}^{(2)} - X_{t_2}^{(2)}\| \geq \varepsilon) \to 0\]
for every $\varepsilon > 0$.
We use the decomposition
\begin{align}
&X_{t_1}^{(2)} - X_{t_2}^{(2)} =\notag\\
&\quad = \int\limits_{\|x\| \geq 1}\int\limits_{M_d^-}\int\limits_{(t_1, t_2]} g_{t_1, t_2} (A,s) \;x\; N (dx, dA, ds) + \int\limits_{\|x\| \geq 1}\int\limits_{M_d^-}\int\limits_{(t_1, t_2]^c} g_{t_1, t_2} (A,s) \;x\; N (dx, dA, ds) \notag \\
&\quad =: Z_{t_1, t_2}^{(1)} + Z_{t_1, t_2}^{(2)}\label{eqn:YZplusZ}
\end{align}
which yields
\begin{align}\label{eqn:PAbsch2}
&n P \Big( a_n^{-1} \sup\limits_{t_1, t_2 \in [0,\delta)} \; \|X_{t_1}^{(2)} - X_{t_2}^{(2)}\| \geq \varepsilon \Big) \leq\notag\\
&\zweiquad\leq n P \Big( a_n^{-1} \sup\limits_{t_1, t_2 \in [0,\delta)} \; \|Z_{t_1, t_2}^{(1)}\| \geq \varepsilon / 2\Big) + n P \Big( a_n^{-1} \sup\limits_{t_1, t_2 \in [0,\delta)} \; \|Z_{t_1, t_2}^{(2)}\| \geq \varepsilon / 2 \Big).
\end{align}
With ${\nu_2} = \nu |_{B_1(0)^c}$ and using the transformation $T: \reald \to \real$ given by $T(x) = \|x\|$ together with the boundedness $f(A,s) \leq C$ for all $(A, s) \in M_d^- \times \real$ we can now calculate
\begin{align}\label{eqn:Z1Absch}
\| Z_{t_1, t_2}^{(1)} \| &\leq  \int\limits_{\| x \| \geq 1}\int\limits_{M_d^-}\int\limits_{(t_1,t_2]}  \| g_{t_1, t_2} (A,s)\| \,\|x\|\, N(dx, dA, ds)\notag\\
&\leq 2\; C\; \Lambda^{(2,T)} (M_d^- \times (t_1,t_2])\notag\\
&= 2\; C\; (L^{(2,T)}_{t_2} - L^{(2,T)}_{t_1}),
\end{align}
where $\Lambda^{(2,T)}$ is a L\'evy basis with generating quadruple $(0, 0, {\nu_2}^T  , \pi)$ and the transformed L\'evy measure ${\nu_2}^T$ is given by ${\nu_2}^T (\cdot) = {\nu_2} (T^{-1} (\cdot))$. By $(L_t^{(2,T)})$ we denote the underlying Lévy process given by $L_t^{(2,T)} = \Lambda^{(2,T)} (M_d^- \times (0,t])$ for $t > 0$. Using a continuous mapping argument similar to Theorem \ref{thm:dd:contmapD} we see that $\nu \in \rvan[\mu_\nu]$ implies ${\nu_2}^T \in \rvan[\mu_{\nu^T}]$ with $\mu_{\nu^T}$ defined respectively. Thus by Proposition 3.1 in \cite{hult:lindskog:2006a} $L_1^{(2,T)} \in \rvan[\mu_{\nu^T}]$ and then by Example \ref{ex:regvarLevyD} also $(L_t^{(2,T)}) \in \rvdd[\schl{\mu}]$ for some measure $\schl{\mu}$. Now another application of Theorem \ref{thm:dd:charregvarD} yields that condition \eqref{eqn:dd:relcomp1} holds for the process $(L_t^{(2,T)})$ and hence, as $\delta \to 0$,
\begin{align*}
\lim\limits_{n \to \infty}n P \Big( a_n^{-1} \sup\limits_{t_1, t_2 \in [0,\delta)} \|Z_{t_1, t_2}^{(1)}\| \geq \varepsilon / 2 \Big) &\leq \lim\limits_{n \to \infty} n P \Big( a_n^{-1} \sup\limits_{t_1, t_2 \in [0,\delta)} (L^{(2,T)}_{t_2} - L^{(2,T)}_{t_1}) \geq \varepsilon /  (4 C) \Big)\\&\to 0.
\end{align*}
Similarly, the supremum of the second term $Z_{t_1, t_2}^{(2)}$ can be bounded by
\begin{align*}
\sup\limits_{t_1, t_2 \in [0,\delta)} \;\| Z_{t_1, t_2}^{(2)} \| &\leq  \int\limits_{\|x\| \geq 1}\int\limits_{M_d^-}\int\limits_{\real} \sup\limits_{t_1, t_2 \in [0,\delta)} \| g_{t_1, t_2} (A,s)\|\,\id_{(t_1, t_2]^c}(s) \,\|x\|\, N(dx, dA, ds)\\
&\leq \int\limits_{\|x\| \geq 1}\int\limits_{M_d^-}\int\limits_{\real} f_\delta (A,s) \,\|x\|\, N(dx, dA, ds)\\
&= \int\limits_{M_d^-}\int\limits_{\real} f_\delta (A,s) \Lambda^{(2,T)} (dA, ds) =: Y.
\end{align*}
Then assumption \eqref{eqn:funcregvarmmabed} implies $f_\delta \in \mathbb{L}^\alpha$ for some $\delta > 0$ sufficiently small and another application of Theorem \ref{thm:dd:regvarMMA} yields $Y \in \rvan[\mu_{Y}]$ with
\[\mu_{Y} (B) := \int\limits_{M_d^-} \int\limits_{\real} \int\limits_{\reald} \id_{B} \left(f_\delta (A,s) \|x\| \right) \mu_\nu (d x) d s \pi (d A).\]
Finally, as $n \to \infty$, we obtain
\begin{align*}
n P \Big( a_n^{-1} \sup\limits_{t_1, t_2 \in [0,\delta)} \; \|Z_{t_1, t_2}^{(2)}\| \geq \varepsilon / 2 \Big) \,&\leq\, n P ( a_n^{-1} \,Y \geq \varepsilon / 2 )\\
&\stackrel{n\to\infty}{\to} \int\limits_{M_d^-} \int\limits_{\real} \mu_\nu ( x: f_\delta (A,s) \|x\| \geq \varepsilon / 2)\; d s\, \pi (d A)\\
&= \mu_\nu ( x: \|x\| \geq \varepsilon / 2) \int\limits_{M_d^-} \int\limits_{\real} f_\delta (A,s)^\alpha\, d s\, \pi (d A)\stackrel{\delta\to0}{\to} 0
\end{align*}
and since $\mu_\nu$ is a Radon measure the result follows by the assumption.\par
\emph{Condition \eqref{eqn:dd:relcomp2}:} The condition follows likewise to condition \eqref{eqn:dd:relcomp1} (note also that the MMA process $(X_t)$ is stationary).\par
\emph{Condition \eqref{eqn:dd:relcomp3}:} For the third condition we 
use \eqref{eqn:YZplusZ} again and obtain
\begin{align*}
&\sup\limits_{t_1\leq t \leq t_2;\; t_2 - t_1 \leq \delta} \min \big\{ \|X_{t}^{(2)} - X_{t_2}^{(2)} \|, \|X_{t_1}^{(2)} - X_{t}^{(2)}\| \big\} \leq\\
&\achtquad \leq \sup\limits_{t_1\leq t \leq t_2;\; t_2 - t_1 \leq \delta} \min \big\{ \|Z_{t, t_2}^{(1)}\|, \|Z_{t_1, t}^{(1)}\| \big\} +  \sup\limits_{t_1 \leq t_2;\; t_2 - t_1 \leq \delta} \|Z_{t_1, t_2}^{(2)}\|
\end{align*}
and
\begin{align*}
&n P \Big( a_n^{-1} \sup\limits_{t_1\leq t \leq t_2;\; t_2 - t_1 \leq \delta} \min \big\{ \|X_{t}^{(2)} - X_{t_2}^{(2)}\|, \|X_{t_1}^{(2)} - X_{t}^{(2)}\|\big\} \geq \varepsilon \Big) \leq\\
&\leq n P \Big( a_n^{-1} \sup\limits_{t_1\leq t \leq t_2; t_2 - t_1 \leq \delta} \min \big\{ \|Z_{t, t_2}^{(1)}\|, \|Z_{t_1, t}^{(1)}\| \big\} \geq \frac{\varepsilon}{2} \Big) + n P \Big( a_n^{-1}  \sup\limits_{t_1 \leq t_2; t_2 - t_1 \leq \delta} \|Z_{t_1, t_2}^{(2)}\| \geq \frac{\varepsilon}{4} \Big).
\end{align*}
Applying \eqref{eqn:Z1Absch} this implies, as $\delta \to 0$,
\begin{align*}
&\lim\limits_{n \to \infty} n P \Big( a_n^{-1} \sup\limits_{t_1\leq t \leq t_2;\; t_2 - t_1 \leq \delta} \min \big\{ \|Z_{t, t_2}^{(1)}\|, \|Z_{t_1, t}^{(1)}\|\big\} \geq \varepsilon / 2 \Big)\\
&\zweiquad\leq \lim\limits_{n \to \infty}n P \Big( a_n^{-1} \sup\limits_{t_1\leq t \leq t_2;\; t_2 - t_1 \leq \delta} \min \big\{ \|L_{t_2}^{(2,T)} - L_{t}^{(2,T)}\|, \|L_{t}^{(2,T)} - L_{t_1}^{(2,T)}\|\big\} \geq \varepsilon / (4\;C)\Big)\\
&\zweiquad \to 0,
\end{align*}
since this is exactly condition \eqref{eqn:dd:relcomp3} for the Lévy process $L_{t}^{(2,T)}$ which is regularly varying in $\dd$ and thus by Theorem \ref{thm:dd:charregvarD} satisfies \eqref{eqn:dd:relcomp3}. Furthermore,
\[\sup\limits_{t_1 \leq t_2;\; t_2 - t_1 \leq \delta} \|Z_{t_1, t_2}^{(2)}\| \leq \int\limits_{M_d^-}\int\limits_{\real} f_\delta (A,s) \Lambda^{(2,T)} (dA, ds) = Y\]
and consequently, as $\delta \to 0$,
\[\lim\limits_{n \to \infty} n P \Big( a_n^{-1}  \sup\limits_{t_1 \leq t_2;\; t_2 - t_1 \leq \delta} \|Z_{t_1, t_2}^{(2)}\| \geq \varepsilon / 4\Big) \leq n P \Big( a_n^{-1} \,Y \geq \varepsilon / 4\Big) \to 0\]
as shown for condition \eqref{eqn:dd:relcomp1}.
\end{proof}

This concludes the proof of Theorem \ref{thm:funcregvarmma}.

\section{Application to SupOU Processes}\label{sec:dd:supOU}

Superpositions of Ornstein-Uhlenbeck processes (supOU processes) have useful properties and a wide range of applications. A supOU process $(X_t)$ can be defined as an MMA process with kernel function
\begin{equation}\label{eqn:supOU}
f(A,s) = e^{As} \id_{[0,\infty)} (s).
\end{equation}
We will shortly recall the main results of \cite{barndorff:stelzer:2011a} and \cite{moser:stelzer:2011}. Sufficient conditions for the existence of supOU processes are given in the following theorem which takes the special properties of supOU processes into account.

\begin{theorem}[\cite{barndorff:stelzer:2011a}, Theorem 3.1]\label{thm:pp:supOUex}
$\;$ Let $X_t$ be an $\reald$-valued supOU process as defined by \eqref{eqn:supOU}. If
\[\int_{\|x\|>1} \ln (\|x\|) \nu (dx) < \infty\]
and there exist measurable functions $\rho: M_d^- \mapsto \real^+ \wo \{0\}$ and $\kappa: M_d^- \mapsto [1, \infty)$ such that
\[\left\| e^{A s} \right\| \leq \kappa(A) e^{- \rho (A) s}\; \forall s \in \real^+\mbox{ $\pi$-almost surely and }  \int\limits_{M_d^-} \frac{\kappa (A)^2}{\rho (A)} \pi (d A) < \infty,\]
then the supOU process $X_t = \int_{M_d^-} \int_{-\infty}^t e^{A (t - s)} \Lambda (dA, ds)$ is well defined for all $t \in \real$ and stationary. Furthermore, the stationary distribution of $X_t$ is infinitely divisible with characteristic triplet $(\gamma_X, \Sigma_X, \nu_X)$ given by Theorem \ref{thm:dd:MMAex}.
\end{theorem}

Conditions for regular variation of $X_t$ and of the finite-dimensional distributions of $(X_t)$ are given by the following result.

\begin{corollary}[\cite{moser:stelzer:2011}, Cor. 4.3 and Cor. 4.6]\label{cor:pp:regvarsupOU}
Let $\Lambda \in \reald$ be a Lévy basis on $M_d^- \times \real$ with generating quadruple $(\gamma, \Sigma, \nu, \pi)$ and let $\nu \in \rvan[\mu_\nu]$. If the conditions of Theorem \ref{thm:pp:supOUex} hold and additionally
\[\int\limits_{M_d^-} \frac{\kappa(A)^{\alpha}}{\rho (A)}\; \pi (d A) < \infty,\]
then $X_0 = \int_{M_d^-}\int_{\real^+} e^{As} \Lambda (dA, ds) \in \rvan[\mu_X]$ with Radon measure
\[\mu_X (\cdot) := \int\limits_{M_d^-} \int\limits_{\real^+} \int\limits_{\reald} \id_{(\cdot)} \left(e^{As}x\right) \mu_\nu (d x) d s \pi (d A).\]
Furthermore, the finite dimensional distributions $( X_{t_1}, \ldots, X_{t_k} )$, $t_i \in \real$ and $k \in \nat$, are also regularly varying with index $\alpha$ and given limiting measure $\mu_{t_1, \ldots, t_k}$.
\end{corollary}

In order to apply Theorem \ref{thm:funcregvarmma} to obtain conditions for regular variation of supOU processes in $\dd$, we state some useful sufficient conditions for the function
\[f_{\delta} (A,s) = \sup_{t_1 \leq t_2;\; t_2 - t_1 \leq \delta} \| f(A, t_2 - s) - f(A, t_1 - s) \|\; \id_{(t_1,t_2]^c}\, (s)\]
to be an element of $\mathbb{L}^\alpha$ for $\alpha > 0$.

\begin{proposition}\label{prop:llsupOU}
Let $f(A,s) = e^{As} \id_{[0,\infty)} (s)$ be the kernel function of a supOU process satisfying the conditions of Theorem \ref{thm:pp:supOUex} and let $f_\delta$ be given by \eqref{eqn:fdelta}. If for some $\alpha > 0$ 
\[\int\limits_{M_d^-} \frac{\kappa(A)^\alpha}{\rho (A)} \;\pi (dA) < \infty,\]
and
\[\int\limits_{M_d^-} \kappa(A)^\alpha\;\pi (dA) < \infty,\]
then $f_\delta \in \mathbb{L}^\alpha$.
\end{proposition}

\begin{proof}
We start with the observation
\begin{align}\label{eqn:fdeltadecompsupOU}
f_\delta (A,s)^\alpha &= \quad\; \sup\limits_{t_1 \leq t_2;\; t_2 - t_1 \leq \delta} \;\big\| e^{A (t_2 - s)} - e^{A (t_1 - s)}\big\|^\alpha \;\id_{(-\infty, t_1]} (s)\;\notag\\
&\leq \quad\; \sup\limits_{t_1 \leq t_2;\; t_2 - t_1 \leq \delta} \;\big\| e^{A (t_2 - s)} - e^{A (t_1 - s)}\big\|^\alpha \;\id_{(-\infty, 0]} (s)\;\notag\\
&\;\;\;\;\;+ \sup\limits_{t_1 \leq t_2;\; t_2 - t_1 \leq \delta} \;\big\| e^{A (t_2 - s)} - e^{A (t_1 - s)}\big\|^\alpha \;\id_{(0, t_1]} (s).
\end{align}
Now for the first summand we obtain
\begin{align*}
&\int\limits_{M_d^-} \int\limits_{\real} \sup\limits_{t_1 \leq t_2;\; t_2 - t_1 \leq \delta}\; \big\| e^{A (t_2 - s)} - e^{A (t_1 - s)} \big\|^\alpha \; \id_{(-\infty, 0]} (s) \;ds\; \pi(dA) \leq\\
&\achtquad \leq \int\limits_{M_d^-} \int\limits_{-\infty}^0 \sup\limits_{t_1 \leq t_2;\; t_2 - t_1 \leq \delta}\;  \big( \big\| e^{A (t_2 - s)}\big\| + \big\|e^{A (t_1 - s)}\big\|\big)^\alpha\;ds\; \pi(dA)\\
&\achtquad \leq \int\limits_{M_d^-} \int\limits_{-\infty}^0 \sup\limits_{t_1 \leq t_2;\; t_2 - t_1 \leq \delta}\;  \kappa(A)^\alpha \big(  e^{-\rho (A) (t_2 - s)} + e^{-\rho (A) (t_1 - s) } \big)^\alpha\;ds\; \pi(dA)\\
&\achtquad \leq 2^\alpha\; \int\limits_{M_d^-} \int\limits_{-\infty}^0 \kappa(A)^\alpha \;e^{s \rho (A) \alpha} \;ds \pi(dA)\\
&\achtquad = \frac{2^\alpha}{\alpha} \; \int\limits_{M_d^-}  \frac{\kappa(A)^\alpha}{\rho (A)} \; \pi(dA)\\
&\achtquad < \infty
\end{align*}
and the second summand of \eqref{eqn:fdeltadecompsupOU} can be bounded by
\begin{align*}
&\int\limits_{M_d^-} \int\limits_{\real} \sup\limits_{t_1 \leq t_2;\; t_2 - t_1 \leq \delta}\;  \big\| e^{A (t_2 - s)} - e^{A (t_1 - s)} \big\|^\alpha \; \id_{(0, t_1]} (s) \;ds\; \pi(dA) \leq\\
&\vierquad\zweiquad \leq \int\limits_{M_d^-} \int\limits_{\real} \sup\limits_{t_1 \leq t_2;\; t_2 - t_1 \leq \delta}\;  \big( \big\| e^{A (t_2 - s)}\big\| + \big\|e^{A (t_1 - s)}\big\|\big)^\alpha\; \id_{(0, t_1]} (s) \;ds\; \pi(dA)\\
&\vierquad\zweiquad  \leq \int\limits_{M_d^-} \int\limits_{\real} \sup\limits_{t_1 \leq t_2;\; t_2 - t_1 \leq \delta}\;  \kappa(A)^\alpha \big(  e^{-\rho (A) (t_2 - s)} + e^{-\rho (A) (t_1 - s) } \big)^\alpha\; \id_{(0, t_1]} (s) \;ds\; \pi(dA)\\
&\vierquad\zweiquad \leq  \int\limits_{M_d^-} \int\limits_{\real} \kappa(A)^\alpha \;2^\alpha\; \id_{[0, 1]} (s)  \;ds \pi(dA)\\
&\vierquad\zweiquad = 2^\alpha\; \int\limits_{M_d^-}  \kappa(A)^\alpha \; \pi(dA)\\
&\vierquad\zweiquad < \infty.
\end{align*}
\end{proof}

\begin{remark}\label{rem:kappabounded}
Requiring $f_\delta \in \mathbb{L}^\alpha$, the first condition
\[\int\limits_{M_d^-} \frac{\kappa(A)^\alpha}{\rho (A)} \;\pi (dA) < \infty\]
is already needed for the regular variation condition $f \in \mathbb{L}^\alpha$ of Corollary \ref{cor:pp:regvarsupOU}. The additional condition $\int_{M_d^-} \kappa(A)^\alpha \;\pi (dA) < \infty$ is of importance if $\rho (A)$ decays very fast. In that case
\[\frac{\kappa (A)^\alpha}{\rho (A)} \ll \kappa (A)^\alpha\]
for values of $A \in M_d^-$ with high norm and therefore, the stronger integrability condition is needed. If $\kappa$ is bounded, then so is the mixed moving average kernel function and $f_\delta$ is also bounded. Moreover, the condition
\[\int\limits_{M_d^-} \kappa(A)^\alpha\;\pi (dA) < \infty\]
is true for all $\alpha > 0$ and for all probability measures $\pi$.
\end{remark}

%
%

Now we can use Proposition \ref{prop:llsupOU} to obtain conditions for functional regular variation of supOU processes with sample paths in $\dd$. Therefore, we restrict the time interval to $t \in [0,1]$ and assume the supOU process to have c\`adl\`ag sample paths, see Section \ref{sec:dd:sample} and \cite{barndorff:stelzer:2011a}, Theorem 3.12, for details on the sample path behavior of supOU processes.

\begin{theorem}\label{thm:funcregvarsupOU}
Let $\Lambda$ and $\Lambda_2$ be $\reald$-valued L\'evy bases on $M_d^- \times \real$ with generating quadruples $(\gamma, \Sigma, \nu, \pi)$ and $(0, 0, \nu |_{B_1(0)^c}, \pi)$ respectively such that $\nu \in \rvan[\mu_\nu]$. Assume that the supOU process $(X_t)$ given by $X_t = \int_{M_d^-}\int_{-\infty}^t e^{A(t-s)} \Lambda (dA, ds)$ exists for $t \in [0,1]$ (in the sense of Theorem \ref{thm:pp:supOUex}) and that the processes $X_t$ and $X_t^{(2)} = \int_{M_d^-}\int_{\real} f(A,t-s) \Lambda_2 (dA, ds)$ have c\`adl\`ag sample paths. Furthermore, suppose that $\kappa$ is bounded and that 
\[\int\limits_{M_d^-} \frac{\kappa(A)^{\alpha}}{\rho (A)}\; \pi (d A) < \infty.\]
If $f(A,s) = e^{As} \id_{[0,\infty)} (s)$ and the function $f_\delta$ given by \eqref{eqn:fdelta} satisfy condition \eqref{eqn:xt2ex},
then 
\[(X_t)_{t \in [0,1]} \in \rvdd ,\]
where $\mu$ is uniquely determined by the measures $\mu_{t_1, \ldots, t_k}$ in Corollary \ref{cor:pp:regvarsupOU}.
\end{theorem}


\begin{proof}
Applying Theorem \ref{thm:funcregvarmma}, we start by observing that the supOU kernel function
\[f(A,s) = e^{As} \id_{[0,\infty)} (s) \leq \kappa (A)\]
is bounded if $\kappa$ is.
Next we show that, as $\delta \to 0$, 
\[\int\limits_{M_d^-} \int\limits_{\real} f_\delta (A,s)^\alpha\, d s \pi (d A) \to 0.\]
The assumptions together with Proposition \ref{prop:llsupOU} and Remark \ref{rem:kappabounded} yield $f_\delta \in \mathbb{L}^\alpha$ and thus by Lemma \ref{lemma:kvgzpunkwint} it is sufficient to show that $\;\lim_{\delta \to 0} \; f_\delta (A,s) \to 0$ for $\pi \times \lambda$ almost every $(A,s)$. When considering differences of matrix exponentials, we can use the inequality
\begin{align*}
\big\| e^{A (t_2 - s)} - e^{A (t_1 - s)}\big\| &= \bigg\| \sum\limits_{k=0}^{\infty} \frac{A^k ((t_2 - s)^k - (t_1 - s)^k)}{k !}\bigg\| \leq \sum\limits_{k=0}^{\infty} \frac{\|A\|^k ((t_2 - s)^k - (t_1 - s)^k)}{k !}\\
&= e^{\|A\| (t_2 - s)} - e^{\|A\| (t_1 - s)}
\end{align*}
if $t_2 > t_1$. This yields
\begin{align*}
f_\delta (A,s) &=\sup\limits_{t_1 \leq t_2;\; t_2 - t_1 \leq \delta}\; \big\| e^{A (t_2 - s)} - e^{A (t_1 - s)}\big\|\;\id_{(-\infty, t_1]} (s)\\
& \leq \sup\limits_{x \in [0, 1 - \delta]}\;\sup\limits_{t_2, t_1 \in [x, x + \delta];\; t_1 \leq t_2} \; \left( e^{\|A\| (t_2 - s)} - e^{\|A\| (t_1 - s)}\right) \;\id_{(-\infty, t_1]} (s)\\
&\leq \sup\limits_{x \in [0, 1 - \delta]}\; \left( e^{\|A\| (x + \delta - s)} - e^{\|A\| (x - s)}\right) \;\id_{(-\infty, t_1]} (s)\\
& \leq \sup\limits_{x \in [0, 1 - \delta]}\; \left( e^{\|A\| \delta} - 1\right) e^{\|A\| (x  - s)} \;\id_{(-\infty, t_1]} (s)\\
& \leq \left( e^{\|A\| \delta} - 1\right) e^{\|A\| (1 - \delta - s)} \;\id_{(-\infty, 1 - \delta]} (s)
\end{align*}
and by the continuity of the exponential this term converges to $0$ as $\delta \to 0$ for every $(A,s) \in \borel (M_d^- \times \real)$.
\end{proof}

Conditions for $f$ and $f_\delta$ to satisfy the existence condition \eqref{eqn:xt2ex}, i.e.
\[\int \limits_{M_d^-} \int \limits_{\real} \int \limits_{\|x\| > 1} (1 \wedge \| f(A, s) x\|) \,\nu (dx)\,ds\,\pi (dA) < \infty\]
can be obtained by combining Proposition \ref{prop:llsupOU} with Lemma \ref{lemma:charxt2ex}.

\begin{corollary}\label{lemma:charxt2exsupOU}
Let $f(A,s) = e^{As} \id_{[0,\infty)} (s)$ be the kernel function of a supOU process satisfying the conditions of Theorem \ref{thm:pp:supOUex} and let $f_\delta$ be given by \eqref{eqn:fdelta}. Then $f$ and $f_\delta$ satisfy condition \eqref{eqn:xt2ex} if one of the following two conditions are satisfied:
\begin{enumerate}
	\item[(i)] $\alpha > 1$ as well as
		\[\int\limits_{M_d^-} \frac{\kappa(A)}{\rho (A)} \;\pi (dA) < \infty \quad \mbox{ and } \int\limits_{M_d^-} \kappa(A) \;\pi (dA) < \infty.\]
	\item[(ii)] $\alpha \leq 1$ and there exists $\varepsilon \in (0,\alpha)$ such that
		\[\int\limits_{M_d^-} \frac{\kappa(A)^{\alpha - \varepsilon}}{\rho (A)} \;\pi (dA) < \infty \quad \mbox{ and } \int\limits_{M_d^-} \kappa(A)^{\alpha - \varepsilon} \;\pi (dA) < \infty.\]
\end{enumerate}
\end{corollary}

In correspondence with Remark \ref{rem:kappabounded} we mention that for all $\alpha > 0$, $\varepsilon \in (0,\alpha)$ and for all probability measures $\pi$ the conditions 
\[\int\limits_{M_d^-} \kappa(A) \;\pi (dA) < \infty \quad \mbox{ and } \int\limits_{M_d^-} \kappa(A)^{\alpha - \varepsilon} \;\pi (dA) < \infty\]
are redundant if $\kappa$ is bounded.

\section{Point Process Convergence}\label{sec:dd:pointprocess}

In this section we discuss the use of the results of the previous two sections in combination with point process results for stochastic processes with sample paths in $\dd$. Therefore, let $M_p (\ddcon)$ denote the space of all point measures on $\ddcon$ equipped with the $\what$-topology and let $\varepsilon_x$ be the Dirac measure at the point $x$. Furthermore, let $X_i$, $i \in \nat$, be a sequence of iid copies of a regularly varying stochastic process $X \in \rvdd$ with values in $\dd$.\par
We start by stating the main result that links regular variation of $X$ to weak convergence of the point processes
\[N_n = \sum\limits_{i = 1}^n \varepsilon_{a_n^{-1} X_i}, \quad n \in \nat.\]
The following theorem is the extension of the classical result of Proposition 3.21 in \cite{resnick:1987} to a state space which is not locally compact. Similar results have also been proved by \cite{dehaan:lin:2001}, Theorem 2.4, in the case of real-valued processes which are regularly varying with index $1$ and by \cite{davis:mikosch:2008} for $\dd$-valued random fields.

\begin{theorem}\label{thm:dd:mainpp}
Let $(X_i)_{i \in \nat}$ be an iid sequence of stochastic processes with values in $\dd$. Then $X_1 \in \rvdd$ if and only if $N_n \gegend N$ in $M_p (\ddcon)$, where $N$ is a Poisson random measure with mean measure $\mu$ (short PRM($\mu$)).
\end{theorem}

\begin{proof}
The proof can be obtained by changing from vague-topology to the $\what$-topology in the proof of Proposition 3.21 in \cite{resnick:1987}. This change of topology does not affect the proof which is based on the Laplace functionals of the point processes involved (cf.\ \cite{davis:mikosch:2008}, Proof of Lemma 2.2).
\end{proof}

This result can now be combined with the results of Sections \ref{sec:dd:funcregvar} and \ref{sec:dd:supOU} to obtain functional point process convergence for MMA and supOU processes. Point processes of that kind include full information of the complete paths of the process $X$. In combination with the continuous mapping theorem (cf.\ \cite{daley:verejones:1988}, Proposition A2.3.V) this is an extremely powerful tool to analyze the extremal behavior of MMA and supOU processes. Using such methods, one gets a better understanding of the structure of the extreme values and their properties, e.g.\ the extremal clustering behavior or long memory effects.\par

In contrast to finite-dimensional point process results, functional point process convergence does not only allow to analyze, for example, the behavior of maxima at fixed time points, but also of functionals acting on the paths of the process in compact time intervals. Examples of such functionals are the subadditive functionals (e.g.\ suprema) studied by \cite{rosinski:samorodnitsky:1993} for a subexponential, by \cite{braverman:samorodnitsky:1995} for an exponential, and by \cite{braverman:mikosch:samorodnitsky:2002} for a univariate regularly varying setting. Moreover, since point processes of suprema do not incorporate the directions of the extremes, it is also possible to include the directions into the analyzed point processes. Regarding the above issues introductions to the use of point processes in extreme value theory can be found in \cite{embrechts:klueppelberg:mikosch:1997}, \cite{resnick:1987}, \cite{resnick:2007}, \cite{leadbetter:lindgren:rootzen:1983} and \cite{dehaan:ferreira:2006} and for the exemplary use of functional point processes, see \cite{dehaan:lin:2001} and \cite{davis:mikosch:2008}.

A thorough investigation of all these issues for mixed moving average processes is beyond the scope of the present paper and the topic of future research.\par


\begin{thebibliography}{10}

\bibitem{barndorff:2001}
O.E. Barndorff-Nielsen.
\newblock Superposition of \textsc{O}rnstein-\textsc{U}hlenbeck type processes.
\newblock {\em Theory Probab. Appl.}, 45:175--194, 2001.

\bibitem{barndorff:shephard:2001}
O.E. Barndorff-Nielsen and N.~Shephard.
\newblock Non-\textsc{G}aussian \textsc{O}rnstein-\textsc{U}hlenbeck-based
  models and some of their uses in financial economics (with discussion).
\newblock {\em J. R. Stat. Soc. Ser. B Stat. Methodol.}, 63:167--241, 2001.

\bibitem{barndorff:stelzer:2011a}
O.E. Barndorff-Nielsen and R.~Stelzer.
\newblock Multivariate sup\textsc{OU} processes.
\newblock {\em Ann. Appl. Probab.}, 21(1):140--182, 2011.

\bibitem{basse:pedersen:2009}
A.~Basse and J.~Pedersen.
\newblock \textsc{L}\'evy driven moving averages and semimartingales.
\newblock {\em Stoch. Proc. Appl.}, 119:2970--2991, 2009.

\bibitem{basse:rosinski:2011}
A.~Basse and J.~Rosi\'nski.
\newblock On the uniform convergence of random series in {S}korohod space and
  representations of c\`adl\`ag infinitely divisible processes.
\newblock Research report, 2011.
\newblock arXiv:1111.1682.

\bibitem{BasseRosinski2012}
Andreas Basse-O'Connor and Jan Rosi\'nski.
\newblock Characterization of the finite variation property for a class of
  stationary increment infinitely divisible processes.
\newblock Research report, 2012.
\newblock arXiv:1201.4366v1.

\bibitem{bender:lindner:schicks:2011}
C.~Bender, A.~Lindner, and M.~Schicks.
\newblock Finite variation of fractional \textsc{L}évy processes.
\newblock {\em J. Theor. Probab., accepted for publication}, 2011.
\newblock DOI: 10.1007/s10959-010-0339-y.

\bibitem{billingsley:1968}
P.~Billingsley.
\newblock {\em Convergence of Probability Measures}.
\newblock Wiley, New York, 1968.

\bibitem{bingham:goldie:teugels:1987}
N.~H. Bingham, C.~M. Goldie, and J.~L. Teugels.
\newblock {\em Regular Variation}.
\newblock Cambridge University Press, Cambridge, 1987.

\bibitem{braverman:mikosch:samorodnitsky:2002}
M.~Braverman, T.~Mikosch, and G.~Samorodnitsky.
\newblock Tail probabilities of subadditive functionals acting on
  \textsc{L}\'evy processes.
\newblock {\em Ann. Appl. Probab.}, 12:69--100, 2002.

\bibitem{braverman:samorodnitsky:1995}
M.~Braverman and G.~Samorodnitsky.
\newblock Functionals of infinitely divisible stochastic processes with
  exponential tails.
\newblock {\em Stoch. Proc. Appl.}, 56:207--231, 1995.

\bibitem{brockwell:2004}
P.~Brockwell.
\newblock Representations of continuous-time \textsc{ARMA} processes.
\newblock {\em J. Appl. Probab.}, 41:375--382, 2004.

\bibitem{cambanis:nolan:rosinski:1990}
S.~Cambanis, J.~P. Nolan, and J.~Rosi\'nski.
\newblock On the oscillation of infinitely divisible and some other processes.
\newblock {\em Stoch. Proc. Appl.}, 35:87--97, 1990.

\bibitem{daley:verejones:1988}
D.~J. Daley and D.~Vere-Jones.
\newblock {\em An Introduction to the Theory of Point Processes}.
\newblock Springer, New York, 1988.

\bibitem{davis:mikosch:2008}
R.~A. Davis and T.~Mikosch.
\newblock Extreme value theory for space-time processes with heavy-tailed
  distributions.
\newblock {\em Stoch. Proc. Appl.}, 118:560--584, 2008.

\bibitem{dehaan:ferreira:2006}
L.~de~Haan and A.~Ferreira.
\newblock {\em Extreme Value Theory}.
\newblock Springer, New York, 2006.

\bibitem{dehaan:lin:2001}
L.~de~Haan and T.~Lin.
\newblock On convergence toward an extreme value distribution in
  \textsc{$C$}$[0,1]$.
\newblock {\em Ann. Probab.}, 29(1):467--483, 2001.

\bibitem{embrechts:klueppelberg:mikosch:1997}
P.~Embrechts, C.~Kl\"uppelberg, and T.~Mikosch.
\newblock {\em Modelling Extremal Events for Insurance and Finance}.
\newblock Springer, Berlin, 1997.

\bibitem{fasen:2005}
V.~Fasen.
\newblock Extremes of regularly varying \textsc{L}évy driven mixed moving
  average processes.
\newblock {\em Adv. Appl. Probab.}, 37:993--1014, 2005.

\bibitem{fasen:2009}
V.~Fasen.
\newblock Extremes of \textsc{L}\'evy driven mixed \textsc{MA} processes with
  convolution equivalent distributions.
\newblock {\em Extremes}, 12(3):265--296, 2009.

\bibitem{fasen:klueppelberg:2007}
V.~Fasen and C.~Klüppelberg.
\newblock Extremes of sup\textsc{OU} processes.
\newblock In F.E. Benth, G.~Di Nunno, T.~Lindstrom, B.~{\O}ksendal, and
  T.~Zhang, editors, {\em Stochastic Analysis and Applications: The Abel
  Symposium 2005}, pages 340--359. Springer, New York, 2007.

\bibitem{hult:lindskog:2005}
H.~Hult and F.~Lindskog.
\newblock Extremal behavior of regularly varying stochastic processes.
\newblock {\em Stoch. Proc. Appl.}, 115:249--274, 2005.

\bibitem{hult:lindskog:2006a}
H.~Hult and F.~Lindskog.
\newblock On regular variation for infinitely divisible random vectors and
  additive processes.
\newblock {\em Adv. Appl. Probab.}, 38:134--148, 2006.

\bibitem{hult:lindskog:2006b}
H.~Hult and F.~Lindskog.
\newblock Regular variation for measures on metric spaces.
\newblock {\em Publications de l'Institut Mathématique, Nouvelle Série},
  80:121--140, 2006.

\bibitem{kallenberg:1983}
O.~Kallenberg.
\newblock {\em Random Measures}.
\newblock Academic Press, New York, 1983.

\bibitem{leadbetter:lindgren:rootzen:1983}
M.~R. Leadbetter, G.~Lindgren, and H.~Rootz\'en.
\newblock {\em Extremes and Related Properties of Random Sequences and
  Processes}.
\newblock Springer, Berlin, 1983.

\bibitem{lindskog:2004}
F.~Lindskog.
\newblock {\em Multivariate Extremes and Regular Variation for Stochastic
  Processes}.
\newblock PhD thesis, ETH Zurich, 2004.

\bibitem{marcus:rosinski:2005}
M.~B. Marcus and J.~Rosi\'nski.
\newblock Continuity and boundedness of infinitely divisible processes: A
  {P}oisson point process approach.
\newblock {\em J. Theor. Probab.}, 18:109--160, 2005.

\bibitem{marquardt:2006}
T.~Marquardt.
\newblock Fractional \textsc{L}évy processes with an application to long memory
  moving average processes.
\newblock {\em Bernoulli}, 12:1009--1126, 2006.

\bibitem{marquardt:2007}
T.~Marquardt.
\newblock Multivariate \textsc{FICARMA} processes.
\newblock {\em J. Mult. Anal.}, 98:1705--1725, 2007.

\bibitem{maruyama:1970}
F.~Maruyama.
\newblock Infinitely divisible processes.
\newblock {\em Theory Probab. Appl.}, 15:1--22, 1970.

\bibitem{moser:stelzer:2011}
M.~Moser and R.~Stelzer.
\newblock Tail behavior of multivariate \textsc{L}évy driven mixed moving
  average processes and related stochastic volatility models.
\newblock {\em Adv. Appl. Prob.}, 43(4):1109--1135, 2011.

\bibitem{pedersen:2003}
J.~Pedersen.
\newblock The \textsc{L}évy-\textsc{I}to decomposition of an independently
  scattered random measure.
\newblock {\em MaPhySto Research Report 2, MaPhySto, {\AA}rhus. Available
  online at www.maphysto.dk}, 2003.

\bibitem{pigorsch:stelzer:2009}
C.~Pigorsch and R.~Stelzer.
\newblock A multivariate \textsc{O}rnstein-\textsc{U}hlenbeck type stochastic
  volatility model.
\newblock {\em Submitted for publication}, 2009.

\bibitem{rajput:rosinski:1989}
B.~S. Rajput and J.~Rosi{\'n}ski.
\newblock Spectral representations of infinitely divisible processes.
\newblock {\em Probab. Theory Relat. Fields}, 82:451--487, 1989.

\bibitem{resnick:1987}
S.~I. Resnick.
\newblock {\em Extreme Values, Regular Variation and Point Processes}.
\newblock Springer, New York, 1987.

\bibitem{resnick:2007}
S.~I. Resnick.
\newblock {\em Heavy-Tail Phenomena: Probabilistic and Statistical Modeling}.
\newblock Springer, New York, 2007.

\bibitem{rosinski:1986}
J.~Rosi\'nski.
\newblock On stochastic integral representation of stable processes with sample
  paths in \textsc{B}anach spaces.
\newblock {\em J. Mult. Anal.}, 20:277--302, 1986.

\bibitem{rosinski:1989}
J.~Rosi\'nski.
\newblock On path properties of certain infinitely divisible processes.
\newblock {\em Stoch. Proc. Appl.}, 33:73--87, 1989.

\bibitem{rosinski:samorodnitsky:1993}
J.~Rosi\'nski and G.~Samorodnitsky.
\newblock Distributions of subadditive functionals of sample paths of
  infinitely divisible processes.
\newblock {\em Ann. Probab.}, 21:996--1014, 1993.

\bibitem{rosinski:samorodnitsky:taqqu:1991}
J.~Rosi\'nski, G.~Samorodnitsky, and M.~S. Taqqu.
\newblock Sample path properties of stochastic processes represented as
  multiple stable integrals.
\newblock {\em J. Mult. Anal.}, 37:115--134, 1991.

\bibitem{sato:2002}
K.~Sato.
\newblock {\em Lévy Processes and Infinitely Divisible Distributions}.
\newblock Cambridge University Press, Cambridge, 2002.

\bibitem{surgailis:rosinski:mandrekar:cambanbis:1993}
D.~Surgailis, J.~Rosi{\'n}ski, V.~Mandrekar, and S.~Cambanis.
\newblock Stable mixed moving averages.
\newblock {\em Probab. Theory Relat. Fields}, 97:543--558, 1993.

\end{thebibliography}
\end{document}